\documentclass[reqno]{amsart}

\usepackage[utf8]{inputenc}
\usepackage[T1]{fontenc}
\usepackage{amsthm}
\usepackage{amsmath}
\usepackage{amssymb}
\usepackage{enumitem}
\usepackage[dvips]{graphicx}
\usepackage{color}
\usepackage{hyperref}
\usepackage{url}
\usepackage{stackrel}
\usepackage[left=3.5cm, right=3.5cm, paperheight=11.9in]{geometry}
\usepackage{fancyhdr}
\usepackage{mathrsfs}
\usepackage{stmaryrd}
\usepackage{soul}
\usepackage[normalem]{ulem}
\usepackage{nicefrac}
\usepackage{datetime}
\usepackage{comment}
\longdate

\newtheorem{theorem}{Theorem}
\newtheorem{lemma}{Lemma}
\newtheorem{proposition}{Proposition}
\newtheorem{corollary}{Corollary}
\newtheorem{definition}{Definition}

\theoremstyle{definition}

\newtheorem{example}{\textsc{Example}}
\newtheorem{remark}{\textbf{Remark}}

\newtheorem{question}{\textbf{Question}}

\newtheorem*{claim*}{Claim}

\hypersetup{
    pdftitle={On the notions of upper and lower density},
    pdfauthor={Leonetti and Tringali},
    pdfmenubar=false,
    pdffitwindow=true,
    pdfstartview=FitH,
    colorlinks=true,
    linkcolor=blue,
    citecolor=green,
    urlcolor=cyan
}

\DeclareMathSymbol{\widehatsym}{\mathord}{largesymbols}{"62}

\renewcommand{\bf}{\mathbf}
\renewcommand{\emptyset}{\varnothing}

\renewcommand{\rho}{\varrho}
\renewcommand{\ast}{\star}

\providecommand{\HHb}{\mathbf{H}}

\providecommand{\AAc}{\mathscr{A}}

\providecommand{\bb}{\mathsf{bd}}
\providecommand{\bbf}{\mathfrak{b}}

\providecommand{\CCc}{\mathscr{C}}
\providecommand{\dd}{\mathsf{d}}

\providecommand{\FFc}{\mathscr{F}}

\providecommand{\GGc}{\mathscr{G}}

\providecommand{\HHc}{\mathcal{H}}

\providecommand{\MMc}{\mathcal{M}}

\providecommand{\NNb}{\mathbf{N}}

\providecommand{\PPc}{\mathcal{P}}

\providecommand{\RRb}{\mathbf{R}}

\providecommand{\VVc}{\mathcal{V}}
\providecommand{\WWc}{\mathcal{W}}

\providecommand{\ZZb}{\mathbf{Z}}

\providecommand\dom{{\rm dom}}

\providecommand\llb{\llbracket}
\providecommand\rrb{\rrbracket}

\providecommand{\imag}{{\rm Im}}

\newcommand\pto{\mathrel{\ooalign{\hfil$\mapstochar$\hfil\hfil\cr$\to$\cr}}}

\begin{document}
\title{On the notions of upper and lower density}

\author{Paolo Leonetti}
\address{Department of Statistics, Universit\`a ``Luigi Bocconi'' | via Sarfatti 25, 20136 Milano, Italy}
\curraddr{Institute of Analysis and Number Theory, Graz University of Technology | Kopernikusgasse 24/II, 8010 Graz, Austria}
\email{leonetti.paolo@gmail.com}
\urladdr{https://sites.google.com/site/leonettipaolo/}
\author{Salvatore Tringali}
\address{Department of Mathematics, Texas A\&M University at Qatar | PO Box 23874 Doha, Qatar}
\curraddr{College of Mathematics and Information Science,
	Hebei Normal University | Shijiazhuang, Hebei province, 050000 China}
\email{salvo.tringali@gmail.com}
\urladdr{http://imsc.uni-graz.at/tringali}

%
% 11B05: Density, gaps, topology
% 28A10: Real- or complex-valued set functions
% 39B52: Equations for functions with more general domains and/or ranges
% 39B62: Functional inequalities, including subadditivity, convexity, etc.
% 60B99: None of the above, but in this section (Probability measures)
%

\begin{abstract}
Let $\mathcal{P}({\bf N})$ be the power set of ${\bf N}$. We say that a function $\mu^\ast: \mathcal{P}({\bf N}) \to \bf R$ is an upper density if, for all $X,Y\subseteq{\bf N}$ and $h, k \in {\bf N}^+$, the following hold: (\textsc{f1}) $\mu^\ast({\bf N}) = 1$; (\textsc{f2}) $\mu^\ast(X) \le \mu^\ast(Y)$ if $X \subseteq Y$; (\textsc{f3}) $\mu^\ast(X \cup Y) \le \mu^\ast(X) + \mu^\ast(Y)$; (\textsc{f4}) $\mu^\ast(k\cdot X) = \frac{1}{k} \mu^\ast(X)$, where $k \cdot X := \{kx: x \in X\}$; (\textsc{f5}) $\mu^\ast(X + h) = \mu^\ast(X)$.

We show that the upper asymptotic, upper logarithmic, upper Banach, upper Buck, upper P\'olya, and upper analytic densities, together with all upper $\alpha$-densities (with $\alpha$ a real parameter $\ge -1$), are upper densities in the sense of our definition. Moreover, we establish the mutual independence of axioms (\textsc{f1})-(\textsc{f5}), and we investigate various properties of upper densities (and related functions) under the assumption that (\textsc{f2}) is replaced by the weaker condition that $\mu^\ast(X)\le 1$ for every $X\subseteq{\bf N}$. 

Overall, this allows us to extend and generalize results so far independently derived for some of the classical upper densities mentioned above, thus introducing a certain amount of unification into the theory.
\end{abstract}

\thanks{P.L. was supported by a PhD scholarship from Universit\`a ``L. Bocconi'' and partly by the Austrian Science Fund (FWF), project F5512-N26. S.T. was supported by NPRP grant No. 5-101-1-025 from the Qatar National Research Fund (a member of Qatar Foundation) and partly by the Austrian Science Fund (FWF), Project No. M 1900-N39.}
\maketitle
\thispagestyle{empty}

\section{Introduction}
\label{sec:intro}
Densities have played a fundamental role in the development of (probabilistic and additive)
number theory and certain areas of analysis and ergodic theory,
as witnessed by the great deal of research on the subject.
One reason is that densities provide an
effective alternative to measures when it comes to the problem of studying the interrelation between the ``structure'' of a set of integers $X$ and some kind of information about the ``largeness'' of $X$.

This principle is fully embodied in Erd\H{o}s' conjecture on arithmetic progressions \cite[\S{ }35.4]{Soif} (that any set $X$ of positive integers such that $\sum_{x \in X} \frac{1}{x} = \infty$ contains arbitrarily long finite arithmetic progressions), two celebrated instances of which are Szemer\'edi's theorem on sets of positive upper asymptotic density \cite{Szem} and the Green-Tao theorem on the primes \cite{GrTa}.

The present paper fits into this context,
insofar as we aim to characterize the upper asymptotic and upper Banach densities as two of the uncountably many functions satisfying a suitable set of conditions that we use to give a more conceptual proof of some nontrivial properties of these and many other ``upper densities'' that have often been considered in the literature (see, in particular, \S{ }\ref{sec:examples} and Example \ref{exa:polya}).

An analogous point of view was picked up, for instance, by A.~R.~Freedman and J.~J.~Sember (motivated by the study of convergence in sequence spaces) in \cite{FrSe}, where a \textit{lower density} (on $\NNb^+$) is essentially a non-negative (set) function $\delta_\ast: \PPc(\NNb^+) \to \RRb$ such that, for all $X,Y \subseteq \NNb^+$, the following hold:
%(notation and terminology will be explained later):
\begin{enumerate}[label={\rm (\textsc{l}\arabic{*})}]
	\item\label{item:FS1} $\delta_\ast(X) + \delta_\ast(Y) \le \delta_\ast(X \cup Y)$ if $X \cap Y = \emptyset$;
	\item\label{item:FS2} $\delta_\ast(X) + \delta_\ast(Y) \le 1 + \delta_\ast(X \cap Y)$;
	\item\label{item:FS3} $\delta_\ast(X) = \delta_\ast(Y)$ provided that $|X \triangle Y| < \infty$;
	\item\label{item:FS4} $\delta_\ast(\NNb^+) = 1$.
\end{enumerate}
Then the \textit{upper} density associated to $\delta_\ast$ is the (provably non-negative) function
\begin{equation*}
\label{equ:Freedman-Sember_dual}
\delta^\ast: \PPc(\NNb^+) \to \RRb: X \mapsto 1 - \delta_\ast(\NNb^+ \setminus X),
\end{equation*}
see \cite[\S{ }2]{FrSe}, and it is found that, for all $X,Y \subseteq \NNb^+$, the following hold, see \cite[Proposition 2.1]{FrSe}:
\begin{enumerate}[resume, label={\rm (\textsc{l}\arabic{*})}]
	\item\label{item:FS5} $\delta_\ast(X) \le \delta_\ast(Y)$ and $\delta^\ast(X) \le \delta^\ast(Y)$ if $X \subseteq Y$;
	\item\label{item:FS6} $\delta^\ast(X \cup Y) \le \delta^\ast(X) + \delta^\ast(Y)$;
	\item\label{item:FS7} $\delta^\ast(X) = \delta^\ast(Y)$ provided that $|X \triangle Y| < \infty$;
	\item\label{item:FS8} $\delta_\ast(X) \le \delta^\ast(X)$;
	\item\label{item:FS9} $\delta_\ast(\emptyset) = \delta^\ast(\emptyset) = 0$ and $\delta^\ast(\NNb^+) = 1$.
\end{enumerate}
The basic goal of this paper is actually to give an axiomatization of the notions of upper and lower density (Definitions \ref{def:upperdensity} and \ref{def:lowerdensity}) that is  ``smoother'' than Freedman and Sember's, insofar as it implies some desirable properties that do not necessarily hold for a function $\delta_\ast$ subjected to axioms \ref{item:FS1}-\ref{item:FS4} and for its ``conjugate'' $\delta^\ast$ (see Examples \ref{exa:Freedman-Sember_densities_are_not_scalable} and \ref{exa:freedman&sember}).

Similar goals have been pursued by several authors in the past, though to the best of our knowledge early work on the subject has been mostly focused on the investigation of densities raising as a limit (in a broad sense) of a sequence or a net of measures, see, e.g., R.~C.~Buck \cite{Bu}, R.~Alexander \cite{Ale}, D.~Maharam \cite{Mah}, T.~\v{S}al\'at and R.~Tijdeman \cite{SaTi}, A.~H.~Mekler \cite{Mek}, A. Fuchs and R. Giuliano Antonini \cite{FuGiu}, A.~Blass, R.~Frankiewicz, G.~Plebanek, and C.~Ryll-Nardzewski \cite{BFPR}, M. Sleziak and M. Ziman \cite{SZ}, and M.~Di~Nasso \cite{Nas}.
On the other hand, M. Di Nasso and R. Jin \cite{NJin} have very recently proposed a notion of ``abstract upper density'', which, though much coarser than our notion of upper density, encompasses a significantly larger number of upper (and lower) densities commonly
considered in number theory.
\subsection*{Plan of the paper} In
\S{ }\ref{sec:definitions}, we introduce five axioms
we use to shape the notions of upper and lower density studied in this work, along with the related notion of induced density.
In \S\S{ }\ref{sec:translation_invariance}-\ref{sec:range}, we establish the mutual independence of these axioms,
provide examples of functions that are, or are not, upper or lower densities (in the sense of our definitions),
and prove (Theorem \ref{th:image_of_upper_densities}) that the range of an induced density is the interval $[0,1]$.
In \S{ }\ref{sec:others},
we derive some ``structural properties'': Most notably, we show that the set of all upper densities is convex; has a well-identified maximum (namely, the upper Buck density); and is closed in the topology of pointwise convergence on the set of all functions $\PPc(\NNb^+) \to \RRb$ (from which we obtain, see Example \ref{exa:polya}, that also the upper P\'olya density is an upper density).
Lastly in \S{ }\ref{sec:closing_remarks}, we draw a list of open questions.
\subsection*{Generalities}
\label{sec:notations}
Unless differently specified, the letters $h$, $i$, $j$, $k$, and $l$ (with or without subscripts) will stand for non-negative integers, the letters $m$ and $n$ for positive integers, the letter $p$ for a positive (rational) prime, and the letter $s$ for a positive real number.

We denote by $\NNb$ the set of non-negative integers (so, $0 \in \mathbf N$). For $a,b \in \RRb \cup \{\infty\}$ and $X \subseteq \RRb$ we set
$\llb a, b \rrb := [a,b] \cap \ZZb$ and $X^+ := X \cap {]0,\infty[}$.
Also, we define $\RRb_0^+ := [0,\infty[$.

We let $\HHb$ be either $\ZZb$, $\NNb$, or $\NNb^+$: If, on the one hand, it makes no substantial difference to stick to the assumption that $\HHb = \NNb$ for most of the time, on the other, some statements will be sensitive to the actual choice of $\HHb$ (see, e.g., Example \ref{exa:buck's_measure_density} or Question \ref{open:unique_extension}).

Given $X \subseteq \RRb$ and $h,k \in \RRb$, we define $k \cdot X + h := \{kx+h: x \in X\}$ and take an arithmetic progression of $\HHb$ to be any set of the form $k \cdot \HHb + h$ with $k \in \mathbf N^+$ and $h \in \HHb \cup \{0\}$.
We remark that, in this work, \emph{arithmetic progressions are always infinite, unless noted otherwise}.

For $X,Y \subseteq \HHb$, we define $X^c := \HHb \setminus X$ and $X \triangle Y := (X \setminus Y) \cup (Y \setminus X)$.
Furthermore, we say that a sequence $(x_n)_{n \ge 1}$ is the natural enumeration of a set $X \subseteq \NNb$ if $X = \{x_n: n \in \NNb^+\}$ and $x_n < x_{n+1}$ for each $n \in \NNb^+$. For a set $S$ we let $\PPc(S)$ be the power set of $S$.

Lastly, given a partial function $f$ from a set $X$ to a set $Y$, herein denoted by $f: X \pto Y$, we write $\dom(f)$ for its domain (i.e., the set of all $x \in X$ such that $y=f(x)$ for some $y \in Y$).

Further notations and terminology, if not explained when first introduced, are standard or should be clear from the context.

\section{Upper and lower densities (and quasi-densities)}
\label{sec:definitions}
We will write $\dd_\ast$ and $\dd^\ast$, respectively, for the functions $\PPc(\HHb) \to \RRb$ mapping a set $X \subseteq \HHb$ to its  lower and upper asymptotic (or natural) density, i.e., 
$$
\dd_\ast(X) := \liminf_{n \to \infty} \frac{|X \cap [1, n]|}{n}
$$
and
$$
\dd^\ast(X) := \limsup_{n \to \infty} \frac{|X \cap [1, n]|}{n}.
$$
Moreover, we will denote by $\bb_\ast$ and $\bb^\ast$, respectively, the functions $\PPc(\HHb) \to \RRb$ taking a set $X \subseteq \HHb$ to its lower and upper Banach (or uniform) density, i.e.,
$$
\bb_\ast(X) := \lim_{n \to \infty} \min_{l \ge 0} \frac{|X \cap [l+1, l+n]|}{n}
$$
and
$$
\bb^\ast(X) := \lim_{n \to \infty} \max_{l \ge 0} \frac{|X \cap [l+1, l+n]|}{n}.
$$
The existence of the latter limits, as well as equivalent definitions of $\bb_\ast$ and $\bb^\ast$ are discussed, e.g., in \cite{GTT}.
Note that, although the set $X$ may contain zero or negative integers, the definitions above only involve the positive part of $X$, cf. \cite[p. xvii]{HR} and Example \ref{exa:generalized_asymptotic_upper_densities} below.

It is well known, see, e.g., \cite[\S{ }2]{Niv} and \cite[Theorem 11.1]{NZM}, that if $X$ is an infinite subset of $\NNb^+$ and $(x_n)_{n \ge 1}$ is the natural enumeration of $X$, then
\begin{equation}
\label{equ:low_and_upp_densities_as_alternative_limits}
\dd_\ast(X) = \liminf_{n \to \infty} \frac{n}{x_n}
\quad\text{and}\quad
\dd^\ast(X) = \limsup_{n \to \infty} \frac{n}{x_n}.
\end{equation}
Since $\dd_\ast(X) = \dd^\ast(X) = 0$ for every finite $X \subseteq \NNb^+$, it follows that
$
\dd_\ast(k \cdot X + h) = \frac{1}{k} \dd_\ast(X)$ and $
\dd^\ast(k \cdot  X + h) = \frac{1}{k} \dd^\ast(X)$ for all $X \subseteq \HHb$, $h \in \mathbf N$, and $k \in \mathbf N^+$.

On the other hand, it is straightforward that, for every $X \subseteq \NNb^+$, $h \in \mathbf N$, and $k \in \mathbf N^+$,
\begin{equation*}
\min_{l \ge 0} \bigl|(k \cdot  X + h) \cap [l+1, l+nk]\bigr| = \min_{l \ge 0} \bigl|X \cap [l+1, l+n]\bigr| + O(1)\quad (n \to \infty),
\end{equation*}
which implies $\bb_\ast(k \cdot  X + h) = \frac{1}{k} \bb_\ast(X)$. The same holds for $\bb^\ast$, by similar arguments.

So, based on these considerations, it seems rather natural to forge an abstract notion of upper density that matches the above properties of $\dd^\ast$ and $\bb^\ast$:
\begin{definition}\label{def:upperdensity}
	We say that a function $\mu^\ast: \PPc(\HHb) \to \RRb$ is an \emph{upper density} \textup{(}on $\HHb$\textup{)} if:
	\begin{enumerate}[label={\rm (\textsc{f}\arabic{*})}]
		\item\label{item:F1} $\mu^\ast(\HHb) = 1$;
		\item\label{item:F2} $\mu^\ast(X) \le \mu^\ast(Y)$ for all $X, Y \subseteq \HHb$ with $X \subseteq Y$;
		\item\label{item:F3} $\mu^\ast(X \cup Y) \le \mu^\ast(X) + \mu^\ast(Y)$ for all $X,Y \subseteq \HHb$;
		\item\label{item:F4} $\mu^\ast(k\cdot X) = \frac{1}{k} \mu^\ast(X)$ for all $X \subseteq \HHb$ and $k \in \NNb^+$;
		\item\label{item:F5} $\mu^\ast(X+h) = \mu^\ast(X)$ for all $X \subseteq \HHb$ and $h \in \NNb$.
	\end{enumerate}
	In addition, we call $\mu^\ast$ an \emph{upper quasi-density} \textup{(}on $\HHb$\textup{)} if $\mu^\ast(X) \le 1$ for every $X \subseteq \HHb$ and $\mu^\ast$ satisfies \ref{item:F1} and \ref{item:F3}--\ref{item:F5}.
\end{definition}
Note that \ref{item:F1} and \ref{item:F2} imply that $\mu^\ast(X) \le 1$ for every $X\subseteq \HHb$, which is, however, false if \ref{item:F2} is not assumed (Theorem \ref{th:independence_of_(F2)}): In particular, every upper density is an upper quasi-density.
\begin{definition}\label{def:lowerdensity}
	We take the \emph{dual}, or \emph{conjugate}, of a function $\mu^\ast: \PPc(\HHb) \to \RRb$ to be the map
	$$
	\mu_\ast: \PPc(\HHb) \to \RRb: X \mapsto 1 - \mu^\ast(X^c),
	$$
	and we say that $(\mu_\ast, \mu^\ast)$ is a \emph{conjugate pair} on $\bf H$ if $\mu^\ast(\HHb) = 1$. We refer to $\mu_\ast$ as the \emph{lower dual} of $\mu^\ast$, and reciprocally to $\mu^\ast$ as the \textit{upper dual} of $\mu_\ast$, if $\mu_\ast(X) \le \mu^\ast(X)$ for every $X \subseteq \HHb$.
	
	In particular, we call $\mu_\ast$ the \emph{lower \textup{[}quasi-\textup{]}density} associated with $\mu^\ast$, or alternatively a lower \textup{[}quasi-\textup{]}density \textup{(}on $\mathbf H$\textup{)}, if $\mu^\ast$ is an upper \textup{[}quasi-\textup{]}density.
\end{definition}
By the above, $\dd^\ast$ and $\bb^\ast$ are both upper densities (in the sense of our definitions), and their lower duals are, respectively, $\dd_\ast$ and $\bb_\ast$.
\begin{definition}
	Let $(\mu_\ast, \mu^\ast)$ be a conjugate pair on $\bf H$ and $\mu$ the partial function $\PPc(\HHb) \pto \RRb: X \mapsto \mu^\ast(X)$ whose domain is given by the set
	$$
	\{X \subseteq \HHb: \mu_\ast(X) = \mu^\ast(X)\} = \{X \subseteq \HHb: \mu_\ast(X) + \mu_\ast(X^c) = 1\}.
	$$
	If $\mu^\ast$ is an upper \textup{[}quasi-\textup{]}density, we call $\mu$ the \emph{\textup{[}quasi-\textup{]}density} \textup{(}on $\HHb$\textup{)} \emph{in\-duced by $\mu^\ast$}.
\end{definition}
To ease the exposition, we will say that a function $\mu^\ast: \PPc(\HHb) \to \RRb$ is: \textit{monotone} if it satisfies \ref{item:F2}; \textit{subadditive} if it satisfies \ref{item:F3};
\textit{\textup{(}finitely\textup{)} additive} if $\mu^\ast(X \cup Y) = \mu^\ast(X) + \mu^\ast(Y)$ whenever $X,Y \subseteq \HHb$ and $X \cap Y = \emptyset$; \textit{$(-1)$-homogeneous} if it satisfies \ref{item:F4}; and \textit{trans\-la\-tion\-al invariant}, or \textit{shift-invariant}, if it satisfies \ref{item:F5}. Also, we note that \ref{item:F4} and \ref{item:F5} together are equivalent to:
\begin{enumerate}[label={\rm (\textsc{f}\arabic{*})},resume]
	\setcounter{enumi}{5}
	\item\label{item:F4b} $\mu^\ast(k \cdot X+h) = \frac{1}{k}\mu^\ast(X)$ for all $X \subseteq \HHb$ and $h, k \in \NNb^+$.
\end{enumerate}
Axioms \ref{item:F1}, \ref{item:F2}, and \ref{item:F3} correspond to conditions \ref{item:FS4}, \ref{item:FS5}, and \ref{item:FS6} from \S{ }\ref{sec:intro}, respectively (see also the comments introducing Example \ref{exa:freedman&sember} in \S{ }\ref{sec:others}). Besides that, \ref{item:F1}--\ref{item:F5} encode some of the most desirable features a density ought to have, cf. \cite[\S{ }3]{Grek}: One reason is that, roughly speaking, a density on $\mathbf N^+$ should \textit{ideally} approximate a shift-invariant probability measure with the further property that the measure of $k \cdot \mathbf N^+$ is $\frac{1}{k}$ for all $k \in \mathbf N^+$. But no \textit{countably} additive probability measure with this property can exist, see \cite[Chapter III.1, Theorem 1]{Tene}. Moreover, it is known that the existence of \textit{finitely} additive, shift-invariant probability measures $\mathcal P(\mathbf N^+) \to \RRb$ is not provable in the frame of classical mathematics without (some form of) the axiom of choice (cf. Remark \ref{rem:additivity}), which makes them unwieldy in many practical situations.

The proofs of the following two propositions are left as an exercise for the reader.
\begin{proposition}
	\label{prop:null_invariance}
	Let $\mu^\ast$ be a monotone and subadditive function $\PPc(\HHb) \to \RRb$, and let $X,Y \subseteq \HHb$.
	\begin{enumerate}[label={\rm (\roman{*})}]
		\item\label{prop:null_invariance(i)} If $\mu^\ast(X \triangle Y) = 0$, then
		$\mu^\ast(X) = \mu^\ast(Y)$ and $\mu^\ast(X^c) = \mu^\ast(Y^c)$.
		\item\label{prop:null_invariance(ii)} If $\mu^\ast(Y) = 0$, then $\mu^\ast(X) = \mu^\ast(X \cup Y) = \mu^\ast(X \setminus Y)$.
		\item\label{prop:null_invariance(iii)} If $\mu^\ast(X \setminus Y) = 0$, then $\mu^\ast(X \cap Y) = \mu^\ast(X)$.
		\item\label{prop:null_invariance(iv)} If $\mu^\ast(X) < \mu^\ast(Y)$, then $0 < \mu^\ast(Y) - \mu^\ast(X) \le \mu^\ast(Y \setminus X)$.
	\end{enumerate}
\end{proposition}
\begin{proposition}
	\label{prop:elementary_properties_of_d-pairs}
	Let $(\mu_\ast, \mu^\ast)$ be a conjugate pair on $\HHb$. The following hold:
	\begin{enumerate}[label={\rm (\roman{*})}]
		\item\label{item:prop:elementary_properties_of_d-pairs(i)} $\mu_\ast(\emptyset) = 0$, and $\imag(\mu^\ast) \subseteq [0,1]$ if and only if $\imag(\mu_\ast) \subseteq [0,1]$.
		\item\label{item:prop:elementary_properties_of_d-pairs(iib)} If $\mu^\ast$ is subadditive and $X_1, \ldots, X_n \subseteq \HHb$, then 
		$
		\mu^\ast(X_1 \cup \cdots \cup X_n) \le \sum_{i=1}^n \mu^\ast(X_i)$.
		\item\label{item:prop:elementary_properties_of_d-pairs(iiib)} Let $\mu^\ast$ be $(-1)$-homogeneous. Then $\mu^\ast(\emptyset) = 0$ and $\mu_\ast(\HHb) = 1$.
	\end{enumerate}
	Moreover, for all $X, Y \subseteq \HHb$ we have the following:
	\begin{enumerate}[resume, label={\rm (\roman{*})}]
		\item\label{item:prop:elementary_properties_of_d-pairs(ivb)} Assume $\mu^\ast$ is monotone. If $X \subseteq Y$, then $\mu_\ast(X) \le \mu_\ast(Y)$.
		\item\label{item:prop:elementary_properties_of_d-pairs(vb)} If $\mu^\ast(X \triangle Y) = 0$ and $\mu^\ast$ satisfies \ref{item:F2}-\ref{item:F3}, then 
		$$\mu_\ast(X) = \mu_\ast(Y)
		\quad\text{and}\quad
		\mu_\ast(X^c) = \mu_\ast(Y^c).
		$$
		\item\label{item:prop:elementary_properties_of_d-pairs(vib)} If $\mu^\ast$ is subadditive, then 
		$$
		\max(\mu_\ast(X),0) \le \mu^\ast(X)
		\quad\text{and}\quad
		\mu_\ast(X) + \mu_\ast(Y) \le 1 + \mu_\ast(X \cap Y).
		$$
	\end{enumerate}
\end{proposition}
We continue with a few examples and remarks. In particular, the first example is borrowed from \cite[p. 297]{FrSe}, while the third shows that conditions \ref{item:F1}-\ref{item:F4} alone are not even enough to guarantee that, for a function $\mu^\ast: \PPc(\HHb) \to \RRb$, $\mu^\ast(X) = 0$ whenever $X \subseteq \HHb$ is finite.
\begin{example}
	\label{exa:Freedman-Sember_densities_are_not_scalable}
	Let $\delta_\ast$ be the function $\PPc(\HHb) \to \RRb$ sending a set $X \subseteq \HHb$ either to $1$ if $|X^c| < \infty$ or to $0$ otherwise. It is found that $\delta_\ast$ satisfies axioms \ref{item:FS1}-\ref{item:FS4} and $\delta_\ast(k \cdot \HHb) = 0 \ne \frac{1}{k} \delta_\ast(\HHb)$ for $k \ge 2$; also, if $\delta^\ast$ is the conjugate of $\delta_\ast$, then $\delta^\ast(k \cdot X) = \delta^\ast(X) = 1$ for every $k \ge 1$ and finite $X \subseteq \HHb$. That is, neither $\delta_\ast$ nor $\delta^\ast$ are $(-1)$-ho\-mo\-ge\-ne\-ous.
\end{example}
\begin{example}
	\label{exa:natural_of_uniform_densities}
	We have already noted that $\dd^\ast$ is an upper density, and that $\dd_\ast$ is the dual of $\mathsf d^\ast$ (hence a lower density) and satisfies \ref{item:F1}, \ref{item:F2}, and \ref{item:F4}. Now, if we let
	\begin{equation*}
	X := \bigcup_{n = 1}^\infty \llb (2n-1)!, (2n)!-1 \rrb \subseteq \NNb^+ \quad\text{and}\quad Y := \bigcup_{n = 1}^\infty \llb (2n)!, (2n+1)!-1 \rrb \subseteq \NNb^+,
	\end{equation*}
	then $\dd_\ast(X) = \dd_\ast(Y) = 0$ (e.g., by Lemma \ref{lem:asymptotic_upper_density_of_special_sets} below), while $\mathsf d_\ast(X \cup Y) = 1$. So we see that lower densities (in the sense of our definitions) need not be subadditive: This may sound ``obvious'' (as we mentioned in the introduction that $\mathsf d_\ast$ also satisfies \ref{item:FS6}), but should be confronted with Remark \ref{rem:additivity} below.
\end{example}
\begin{example}
	\label{exa:minimum_function}
	It is straightforward to check (we omit details) that the function
	\begin{equation*}
	\mathfrak{m}: \PPc(\HHb) \to \RRb: X \mapsto \frac{1}{\inf(X^+)},
	\end{equation*}
	where $\inf(\emptyset) := \infty$ and $\frac{1}{\infty} := 0$, satisfies conditions \ref{item:F1}-\ref{item:F4}. Note, however, that $\mathfrak{m}(X+h) \ne \mathfrak{m}(X) \ne 0$ for every $X \subseteq \HHb$ and $h \in \NNb^+$ such that $X^+ \ne \emptyset$.
\end{example}
\begin{remark}
	\label{rem:upper_quasi-densities_are_non-negative}
	Proposition \ref{prop:elementary_properties_of_d-pairs}\ref{item:prop:elementary_properties_of_d-pairs(vib)} yields that an upper quasi-density is necessarily non-negative, hence its range is contained in $[0,1]$.
\end{remark}
\begin{remark}
	\label{rem:Grekos_axioms}
	Four out of the five axioms we are using to shape our notion of upper density are essentially the same as four out of the seven axioms considered in \cite{Grek} as a \textit{suggestion} for an abstract notion of density on $\NNb^+$, see in particular axioms (A2)-(A5) in \cite[\S{ }3]{Grek}.
\end{remark}
\begin{remark}
	\label{rem:additivity}
	All the results in the present paper can be proved by appealing, e.g., to the usual device of Zermelo-Fraenkel set theory \textit{without} the axiom of choice (ZF for short). However, if we assume to work in Zermelo-Fraenkel set theory \textit{with} the axiom of choice (ZFC for short), then it follows from %\cite[\S{ }10]{AM}
	\cite[Appendix 5C]{Dow}
	that there are uncountably many non-negative additive functions $\theta: \PPc(\NNb^+) \to \RRb$ such that $\theta(\NNb^+) = 1$ and $\theta(k \cdot X + h) = \frac{1}{k} \theta(X)$ for all $X \subseteq \NNb^+$ and $h, k \in \NNb^+$,
	whence
	the function $\mu^\ast: \PPc(\HHb) \to \RRb: X \mapsto \theta(X^+)$ is an \textit{additive} upper density.
	
	On the other hand, we have from Proposition \ref{prop:mu_of_finite_sets} below that, if $\mu^\ast$ is an upper quasi-density on $\HHb$, then $\mu^\ast(X) = 0$ for every finite $X \subseteq \HHb$, and it is provable in ZF that the existence of an additive measure $\theta: \PPc(\NNb^+) \to \mathbf R$ which vanishes on singletons yields the existence of a subset of $\RRb$ without the property of Baire, see \cite[\S\S{ }29.37 and 29.38]{Sch96}. Yet, ZF alone does not prove the existence of such an additive measure $\theta$, see \cite{Pinc}.
	
	So, putting it all together, we see that the existence of an additive upper quasi-density on $\HHb$ is provable in ZFC, but independent of ZF. \end{remark}
\begin{remark}
	\label{remark:sigma-subadditivity}
	Axiom \ref{item:F4b} is incompatible with the following condition, which is often referred to as  sigma-subadditivity (in analogy to the sigma-additivity of measures):
	\begin{enumerate}[label={\rm (\textsc{f}\arabic{*})},resume]
		\item\label{item:F3b} If $(X_n)_{n \ge 1}$ is a sequence of subsets of $\HHb$, then $\mu^\ast\Bigl(\bigcup_{n = 1}^\infty X_n\Bigr) \le \sum_{n = 1}^\infty \mu^\ast(X_n)$.
	\end{enumerate}
	Indeed, let $\mu^\ast$ be a function $\mathcal P(\HHb) \to \RRb$ that satisfies \ref{item:F4b}. Then $\mu^\ast$ is $(-1)$-homogeneous, and we find that, if $0 \in \mathbf H$, then $\mu^\ast(\{0\}) = 0$, because $\mu^\ast(\{0\}) = \mu^\ast(k \cdot \{0\}) = \frac{1}{k} \mu^\ast(\{0\})$ for all $k \in \mathbf N^+$. Likewise, we have by \ref{item:F4b} that, for every $k \ge 1$,
	$$
	\mu^\ast(\{1\}) = \mu^\ast(\{1\} + k-1) = \mu^\ast(\{k\}) = \mu^\ast(k \cdot \{1\}) = \frac{1}{k} \mu^\ast(\{1\}),
	$$
	which yields $\mu^\ast(\{1\}) = 0$. So, if $\HHb = \ZZb$ and $k \le 0$, then 
	$$
	\mu^\ast(\{k\}) = \mu^\ast(\{k\} + (-k)) = \mu^\ast(\{0\}) = 0. 
	$$
	It follows that $\mu^\ast(\{k\}) = 0$ for all $k \in \mathbf H$.
	
	With this in hand, let $(x_n)_{n \ge 1}$ be an enumeration of $\HHb$ and set $X_n := \{x_n\}$ for all $n$. Then $\HHb = \bigcup_{n = 1}^\infty X_n$, but $1 = \mu^\ast(\HHb) > \sum_{n = 1}^\infty \mu^\ast(X_n) = 0$, which is incompatible with \ref{item:F3b}.
\end{remark}

\section{Independence of the axioms}
\label{sec:translation_invariance}
The independence of \ref{item:F1} from the other axioms is obvious, while that of \ref{item:F3} and \ref{item:F5} follows from Examples \ref{exa:natural_of_uniform_densities} and \ref{exa:minimum_function}, respectively.
Moreover, \ref{item:F4} is independent from \ref{item:F1}-\ref{item:F3} and \ref{item:F5}: Indeed, these latter conditions are all satisfied by the constant function $\PPc(\HHb) \to \RRb: X \mapsto 1$, which however does not satisfy \ref{item:F4}.
In contrast, the independence of \ref{item:F2} from \ref{item:F1}, \ref{item:F3} and \ref{item:F4b} is much more delicate,
and it clearly follows from the existence of an upper quasi-density that is not an upper density, which is what we are going to prove.

We begin with a rather simple lemma on the upper (and the lower) asymptotic density of subsets of $\NNb^+$ in which ``large gaps'' alternate to ``large intervals'' (we omit the proof).
\begin{lemma}
	\label{lem:asymptotic_upper_density_of_special_sets}
	Let $(a_n)_{n \ge 1}$ and $(b_n)_{n \ge 1}$ be sequences of positive real numbers such that:
	\begin{enumerate}[label={\rm (\roman{*})}]
		\item\label{item:lemma_hypothesis_i} $a_n + 1 \le b_n < a_{n+1}$ for all sufficiently large $n$;
		\item\label{item:lemma_hypothesis_ii} $a_n/b_n \to \ell$ as $n \to \infty$;
		\item\label{item:lemma_hypothesis_iii} $b_n/a_{n+1} \to 0$ as $n \to \infty$.
	\end{enumerate}
	In addition, pick $k \in \NNb^+$ and $h \in \ZZb$, and let $X \subseteq \NNb^+$ be the intersection of the sets $\bigcup_{n=1}^\infty \llb a_n+1, b_n \rrb$ and $k\cdot \HHb+h$. Then $\dd^\ast(X) = \frac{1}{k}(1-\ell)$ and $\dd_\ast(X) = 0$; in particular, if  $\alpha \in [0,1]$ and we let
	$
	a_n := \alpha  (2n-1)! + (1-\alpha) (2n)!$ and $b_n := (2n)!+1$,
	then $\dd^\ast(X) = \frac{\alpha}{k}$ and $\dd_\ast(X) = 0$.
\end{lemma}
The next step is to show how the independence of \ref{item:F2} can be derived from the existence of a suitable ``indexing'' $\iota: \PPc(\HHb) \to \NNb^+ \cup \{\infty\}$ associated with an upper quasi-density $\mu^\ast$ on $\ZZb$ for which $\mu^\ast(\HHb) = 1$ (if $\mathbf H \subsetneq \mathbf Z$, we need not have $\mu^\ast(\mathbf H) = 1$).

More specifically, we will prove that the existence of such an indexing can be used to construct infinitely many non-monotone functions $\mu^\ast: \PPc(\HHb) \to \RRb$ that, on the one hand, are not upper densities (in fact, we will show that they can even be unbounded), and on the other, are subjected to additional constraints on the values they can attain: This will imply, in particular, that \ref{item:F2} is independent of \ref{item:F1}, \ref{item:F3} and \ref{item:F4b} also in the case where $\mu^\ast$ is an upper quasi-density.
\begin{lemma}
	\label{lem:if_an_index_exists}
	Let $\mu^\ast$ be an upper quasi-density on $\ZZb$ with the property that $\mu^\ast(\HHb) = 1$, and suppose there exists a function $\iota: \PPc(\HHb) \to \NNb^+ \cup \{\infty\}$ such that:
	\begin{enumerate}[label={\rm (\textsc{i}\arabic{*})}]
		\item\label{item:I1} $\iota(\HHb) = 1$;
		\item\label{item:I2} $\iota(X) = \infty$ for some $X \subseteq \HHb$ only if $\mu^\ast(X) = 0$;
		\item\label{item:I3} $\iota(Y) \le \iota(X)$ whenever $X \subseteq Y \subseteq \HHb$ and $\mu^\ast(Y) > 0$;
		\item\label{item:I4} $\iota(X) = \iota(k \cdot X + h)$ for all $X \subseteq \HHb$ and $h, k \in \NNb^+$ with $\mu^\ast(X) > 0$.
	\end{enumerate}
	Moreover, let $(a_n)_{n \ge 1}$ be a non-decreasing real sequence with $a_1 = 1$. Then the function
	$$
	\theta^\ast: \PPc(\HHb) \to \RRb: X \mapsto \left\{
	\begin{array}{ll}
	\!\! a_{\iota(X)}  \mu^\ast(X) & \text{if } \iota(X) < \infty \\
	\!\! 0 & \text{otherwise}
	\end{array}
	\!\!\right..
	$$
	satisfies axioms \ref{item:F1}, \ref{item:F3} and \ref{item:F4b}.
	
	Further, if we set $\varrho_n := \sup\{\mu^\ast(X): X \subseteq \HHb \textup{\text{ and }}\iota(X) = n\}$, and, in addition to the previous assumptions, we suppose that
	\begin{enumerate}[resume,label={\rm (\textsc{i}\arabic{*})}]
		\item\label{item:I5} $0 < \varrho_{n+1} \le \varrho_n$ for all $n$, and $\varrho_n < 1$ for all but finitely many $n$;
		\item\label{item:I6} for every $n \in \NNb^+$ and $\varepsilon \in \RRb^+$, there are $X, Y \subseteq \HHb$ with $Y \subsetneq X$ such that $\iota(X) = n$, $\iota(Y) = n+1$, and $\varrho_{n+1} \le (1+\varepsilon) \mu^\ast(Y)$,
	\end{enumerate}
	then, for a given $K \in [1, \infty]$, we can choose the ``weights'' $a_2, a_3, \ldots$ so that $\theta^\ast$ is non-monotone and $\sup_{X  \in \PPc(\HHb)} \theta^\ast(X) = K$.
\end{lemma}
\begin{proof}
	Clearly, $\theta^\ast$ satisfies \ref{item:F1}, since \ref{item:I1} and the hypothesis that $\mu^\ast(\HHb) = a_1 = 1$ give $\theta^\ast(\HHb) = a_1 \mu^\ast(\HHb) = 1$. Also, it is obvious that $\theta^\ast(X) \ge 0$ for all $X \in \mathcal P(\HHb)$.
	
	As for \ref{item:F3}, pick $X,Y \subseteq \HHb$, and assume without loss of generality that $\mu^\ast(X \cup Y) > 0$ and $\mu^\ast(Y) \le \mu^\ast(X)$. We have by \ref{item:I2} and the subadditivity of $\mu^\ast$ that
	$\iota(X) < \infty$  (if $\iota(X) = \infty$, then $\mu^\ast(X \cup Y) \le 2\mu^\ast(X) = 0$, a contradiction). So we obtain from \ref{item:I3} that $\iota(X \cup Y) \le \iota(X) < \infty$, and we conclude that $\theta^\ast$ is subadditive, because
	\begin{equation}
	\label{equ:indicial_inequalities}
	\begin{split}
	\theta^\ast(X \cup Y)
	& = a_{\iota(X \cup Y)} \mu^\ast(X \cup Y) \le a_{\iota(X \cup Y)} \mu^\ast(X) + a_{\iota(X \cup Y)} \mu^\ast(Y) \\
	& \le a_{\iota(X)}  \mu^\ast(X) + a_{\iota(X \cup Y)} \mu^\ast(Y) \le \theta^\ast(X) + \theta^\ast(Y)
	\end{split}
	\end{equation}
	(note that here we have used, among other things, that $a_n \le a_{n+1}$ for all $n$).
	
	Lastly, let $X \subseteq \HHb$ and $h, k \in \NNb^+$. We want to demonstrate that $\theta^\ast(k \cdot X + h) = \theta^\ast(X)$. This is trivial if $\mu^\ast(X) = 0$, since $\theta^\ast(Y) = 0$ for every $Y \subseteq \HHb$ with $\mu^\ast(Y) = 0$, and $\mu^\ast$ being an upper quasi-density implies from \ref{item:F4b} that $\mu^\ast(k \cdot X + h) = \frac{1}{k} \mu^\ast(X) = 0$.
	So assume $\mu^\ast(X) \ne 0$. Then $\iota(k \cdot  X + h) = \iota(X) < \infty$ by \ref{item:I2} and \ref{item:I4}, with the result that
	$$
	\theta^\ast(k\cdot X+h) = a_{\iota(k  \cdot X + h)}  \mu^\ast(k  \cdot X + h) \stackrel{\ref{item:F4b}}{=}
	\frac{1}{k} a_{\iota(X)}  \mu^\ast(X) = \frac{1}{k} \theta^\ast(X).
	$$
	It follows that $\theta^\ast$ satisfies \ref{item:F4b}, which, together with the rest, proves the first part of the lemma.
	
	As for the second part, fix $K \in [1,\infty]$ and let $\iota$ satisfy \ref{item:I5} and \ref{item:I6}. Then we get from \ref{item:I1} and the fact that $\mu^\ast(X) \le \mu^\ast(\mathbf H) = 1$ for all $X \subseteq \mathbf H$, that there exists $v \in \mathbf N_{\ge 2}$ such that
	\begin{equation}
	\label{equ:basic_constraints}
	0 < \varrho_{n+1} \le \varrho_n < \varrho_{v-1} = \cdots = \varrho_1 = 1,\quad\text{for }n \ge v.
	\end{equation}
	So we distinguish two cases depending on the actual value of $K$.
	\begin{description}[leftmargin=0.5cm]
		\item[Case 1] $1 < K \le \infty$. Based on \ref{item:I5}, we make $(a_n)_{n \ge 1}$ into a non-decreasing sequence by letting $a_n := \varrho_n^{-1} \min(2^{n-1}, K)$ (observe that $a_1 = 1$).
		
		Then, given $X \subseteq \HHb$ with $\iota(X) < \infty$, we have $
		\theta^\ast(X) = a_{\iota(X)}  \mu^\ast(X) \le \min(2^{\iota(X)-1}, K) \le K$. In addition, if we set $\delta := \min(2,K)$ and pick $\varepsilon \in {]0,1-\delta^{-1}[}$, we can find $X \subseteq \HHb$ such that $\iota(X) = 2$ and $\mu^\ast(X) \ge (1-\varepsilon)\varrho_2$, so that
		$$
		\theta^\ast(X) = a_2  \mu^\ast(X) \ge (1-\varepsilon)\delta > 1 = \theta^\ast(\HHb).
		$$
		Consequently, we see that $\sup_{X  \in \PPc(\HHb)} \theta^\ast(X) = K$ and $\theta^\ast$ is non-monotone.
		\item[Case 2] $K = 1$. Similarly to the previous case, but now based on \eqref{equ:basic_constraints}, we make $(a_n)_{n \ge 1}$ into a non-decreasing sequence by taking
		$a_1 := \cdots := a_v := 1$ and $a_n := \frac{1}{2}(1+\varrho_v) \varrho_n^{-1}$ for $n > v$.
		
		Then we let $X \subseteq \HHb$ such that $n := \iota(X) < \infty$, so that $\theta^\ast(X) = a_n  \mu^\ast(X)$.
		If $1 \le n \le v$, we have $\theta^\ast(X) = \mu^\ast(X) \in [0,1]$; otherwise,
		$$
		\theta^\ast(X) = \frac{1}{2}(1+\varrho_v) \varrho_n^{-1} \mu^\ast(X) \le \frac{1}{2}(1+\varrho_v) \le 1,
		$$
		since $\iota(X) = n$ yields $\mu^\ast(X) \le \varrho_n$. It follows that $\theta^\ast(X) \le 1 = \theta^\ast(\HHb)$ for every $X \subseteq \HHb$. So $\sup_{X  \in \PPc(\HHb)} \theta^\ast(X) = K = 1$, and we are left to show that $\theta^\ast$ is non-monotone.
		
		For this, note that $\varrho_v < \frac{1}{2}(1 + \varrho_v)$, and let $\varepsilon \in \RRb^+$ be such that $\varrho_v < \frac{1}{2}(1 + \varrho_v)(1+\varepsilon)^{-1}$. By \ref{item:I6}, there exist $X,Y \subseteq \HHb$ with the property that $Y \subsetneq X$, $\iota(X) = v$, $\iota(Y) = v+1$, and $\varrho_{v+1} \le (1+\varepsilon) \mu^\ast(Y)$, from which we see that $\theta^\ast$ is non-monotone, because
		$$
		\theta^\ast(X) = a_v  \mu^\ast(X) = \mu^\ast(X) \le \varrho_v < \frac{1}{2}(1+\varrho_v)(1+\varepsilon)^{-1} \le a_{v+1}  \mu^\ast(Y) = \theta^\ast(Y).
		$$
	\end{description}
	Putting it all together, the proof is thus complete.
\end{proof}
Finally, we prove that there is a function $\iota: \PPc(\HHb) \to \NNb^+ \cup \{\infty\}$ that fulfills the assumptions of Lemma \ref{lem:if_an_index_exists}. This is the content of the following:
\begin{lemma}
	\label{lem:existence_of_an_index}
	Let $\mu^\ast$ be the upper asymptotic density on $\ZZb$ and $\iota$ the function $\PPc(\HHb) \to \NNb^+ \cup \{\infty\}$ taking a set $X \subseteq \HHb$ to the infimum of the integers $n \ge 1$ for which there exists $Y \subseteq \HHb$ such that
	$\mu^\ast(Y) \ge \frac{1}{n}$ and $|(q \cdot Y + r) \setminus X| < \infty$ for some $q \in \NNb^+$ and $r \in \ZZb$,
	with the convention that $\inf(\emptyset) := \infty$. Then $\iota$ satisfies conditions \ref{item:I1}-\ref{item:I6}.
\end{lemma}
\begin{proof}
	To start with, observe that $
	\iota(X) \le \inf\{n \in \NNb^+: \mu^\ast(X) \ge 1/n\}$ for every $X \subseteq \HHb$,
	since $X = 1 \cdot X + 0$ and $|X \setminus  X| = 0 < \infty$. This implies $\iota(\HHb) = 1$, and shows as well that $\iota(X) = \infty$ for some $X \subseteq \HHb$ only if $\mu^\ast(X) = 0$.
	That is, $\iota$ satisfies \ref{item:I1} and \ref{item:I2}.
	
	Next, let $X \subseteq Y \subseteq \HHb$. We claim that $\iota(Y) \le \iota(X)$.
	This is obvious if $\iota(X) = \infty$, so assume $\iota(X) < \infty$. Then, there exists $\WWc \subseteq \HHb$ such that $\mu^\ast(\WWc) \ge 1/\iota(X)$ and $|(q \cdot \WWc + r) \setminus X| < \infty$ for some $q \in \NNb^+$ and $r \in \ZZb$. So we have $|(q \cdot \WWc + r) \setminus Y| < \infty$, because $S \setminus Y \subseteq S \setminus X$ for every $S \subseteq \HHb$ (by the fact that $X \subseteq Y$). This proves that $\iota(Y) \le \iota(X)$, and hence $\iota$ satisfies \ref{item:I3}, since $X$ and $Y$ were arbitrary.
	
	Lastly, we come to \ref{item:I4}. Let $X \subseteq \HHb$ be such that $\mu^\ast(X) > 0$, and pick $h, k \in \NNb^+$. Then it follows from \ref{item:I2} that $\iota(X) < \infty$ and $\iota(k \cdot X + h) < \infty$, because $\mu^\ast(k \cdot X + h) = \frac{1}{k} \mu^\ast(X) \ne 0$ by the fact that $\mu^\ast$ satisfies \ref{item:F4b}. We want to show that $\iota(X) = \iota(k \cdot  X + h)$.
	
	To begin, there exists (by definition) a set $Y \subseteq \HHb$ such that $\mu^\ast(Y) \ge 1/\iota(X)$ and $|(q \cdot Y + r) \setminus X| < \infty$ for some $q \in \NNb^+$ and $r \in \ZZb$. Therefore, we see that $|(\tilde{q} \cdot Y + \tilde{r}) \setminus (k \cdot  X + h)| < \infty$ for $\tilde{q} := qk$ and $\tilde{r} := rk + h$, which, in turn, implies $\iota(k \cdot X + h) \le \iota(X)$.
	
	As for the reverse inequality, there exists (again by definition) a set $Y \subseteq \HHb$ such that $\mu^\ast(Y) \ge 1/\iota(k  \cdot X + h)$ and $|(q \cdot Y + r) \setminus (k \cdot X + h)| < \infty$ for some $q \in \NNb^+$ and $r \in \ZZb$.
	
	Suppose for a contradiction that $\mu^\ast(Y \cap (k\cdot \HHb + l)) < \frac{1}{k} \mu^\ast(Y)$ for every $l \in \llb 0, k-1 \rrb$. Then we get by axioms \ref{item:F3} and \ref{item:F4b} and Proposition \ref{prop:mu_of_finite_sets} below that
	$$
	\mu^\ast(Y) \le \mu^\ast \left(\bigcup_{l=0}^{k-1} (Y \cap (k \cdot \HHb + l))\right)\! \le \sum_{l=0}^{k-1} \mu^\ast(Y \cap (k \cdot \HHb + l)) < \sum_{l=0}^{k-1} \frac{1}{k} \mu^\ast(Y) = \mu^\ast(Y),
	$$
	which is impossible. Accordingly, pick $l \in \llb 0, k-1 \rrb$ such that $\mu^\ast(Y \cap (k\cdot \HHb + l)) \ge \frac{1}{k} \mu^\ast(Y)$, and let $Q \subseteq \HHb$ such that $k \cdot Q + l = Y \cap (k\cdot \HHb + l) \subseteq Y$. We have
	$$
	|(qk \cdot Q + lq + r) \setminus (k \cdot X + h)| \le |(q \cdot Y + r) \setminus (k \cdot X + h)| < \infty,
	$$
	and this entails that for every $y \in Q \setminus S$, where $S \subseteq \HHb$ is a finite set, there is $x \in X$ for which $qky + lq + r - h = kx$, with the result that $\tilde{r} := \frac{1}{k}(lq + r - h) \in \mathbf Z$ and $qy + \tilde{r} = x$.
	
	It follows $|(q\cdot Q + \tilde{r}) \setminus X| < \infty$, which is enough to conclude that $\iota(X) \le \iota(k \cdot X + h)$, since $\frac{1}{k}\mu^\ast(Q) = \mu^\ast(k \cdot Q + l) \ge \frac{1}{k} \mu^\ast(Y)$, and hence $\mu^\ast(Q) \ge \mu^\ast(Y) \ge 1/\iota(k  \cdot X + h)$.
	
	To summarize, we have so far established that $\iota$ satisfies \ref{item:I1}-\ref{item:I4}, and we are left with  \ref{item:I5} and \ref{item:I6}. For this, let $V_\alpha$ be, for each $\alpha \in {]0,1[}$, the set
	\begin{equation}
	\label{equ:V_alpha}
	\bigcup_{n = 1}^\infty \left\llb \alpha  (2n-1)! + (1-\alpha) (2n)!+1, (2n)! + 1 \right\rrb\!,
	\end{equation}
	and for every $n \in \NNb^+$ define $\varrho_n := \{\mu^\ast(X): X \subseteq \HHb \text{ and } \iota(X) = n\}$. It is clear that $\varrho_1 = 1$, and we want to show that $\varrho_n = \frac{1}{n-1}$ for $n \ge 2$. To this end, we need the following:
	\begin{claim*}
		Fix $n \ge 2$ and let $\alpha \in \!{\left[\frac{1}{n}, \frac{1}{n-1}\right[}$. Then $\mu^\ast(V_\alpha) = \alpha$ and $\iota(V_\alpha) = n$.
	\end{claim*}
	\begin{proof}
		If $Y \subseteq \HHb$ and $\mu^\ast(Y) \ge \frac{1}{n-1}$, then $|(q \cdot Y + r) \setminus V_\alpha| = \infty$ for all $q \in \NNb^+$ and $r \in \ZZb$: In fact, we have from axiom \ref{item:F4b} and Lemma \ref{lem:asymptotic_upper_density_of_special_sets} that
		$$
		\mu^\ast(q \cdot Y + r) = \frac{1}{q} \mu^\ast(Y) \ge \frac{1}{(n-1)q} > \frac{\alpha}{q} = \mu^\ast(V_\alpha \cap (q \cdot \NNb + r)),
		$$
		so we get by Proposition \ref{prop:null_invariance}\ref{prop:null_invariance(iv)} that
		\begin{equation*}
		\mu^\ast((q \cdot Y + r) \setminus V_\alpha) = \mu^\ast((q \cdot Y + r) \setminus (V_\alpha \cap (q \cdot \NNb + r))) > 0,
		\end{equation*}
		which is enough to conclude that $(q \cdot Y + r) \setminus V_\alpha$ is an infinite set, see Proposition \ref{prop:mu_of_finite_sets} below, and ultimately shows that $\iota(V_\alpha) \ge n$. On the other hand, a further application of Lemma \ref{lem:asymptotic_upper_density_of_special_sets} yields $\mu^\ast(V_\alpha) = \alpha \ge \frac{1}{n}$, whence $\iota(V_\alpha) \le n$. So putting it all together, the claim is proved.
	\end{proof}
	In fact, the claim implies $\varrho_n = \frac{1}{n-1}$ for $n \ge 2$, since it shows that for every $\varepsilon \in \RRb^+$ we can find $X \subseteq \HHb$ such that $\mu^\ast(X) \ge \frac{1}{n-1} - \varepsilon$ and $\iota(X) = n$, and we have by the observation made at the beginning of the proof that $\varrho_n \le \frac{1}{n-1}$ (otherwise, there would exist $X \subseteq \HHb$ with $\iota(X) = n$ and $\mu^\ast(X) \ge 1/(n-1)$). Therefore, we see that $\iota$ satisfies \ref{item:I5}, since $0 < \varrho_{n+1} < \varrho_n < \varrho_2 = \varrho_1 = 1$ for $n \ge 3$.
	
	As for \ref{item:I6}, note that if $\alpha, \beta \in {]0,1[}$ and $\alpha < \beta$, then for every $n \in \NNb^+$ we have
	\begin{equation*}
	\label{equ:containment}
	(2n-1)! < \beta  (2n-1)! + (1-\beta)  (2n)! < \alpha  (2n-1)! + (1-\alpha)  (2n)! < (2n)!
	\end{equation*}
	and
	$$
	\lim_{n \to \infty} ((\alpha  (2n-1)! + (1-\alpha)  (2n)!) - (\beta  (2n-1)! + (1-\beta)  (2n)!)) = \infty,
	$$
	which yields $V_\alpha \subsetneq V_\beta \subseteq \NNb^+$. Hence, for all $n \in \NNb^+$ and $\varepsilon \in \RRb^+$ there exist $X, Y \subseteq \HHb$ such that $Y \subsetneq X$, $\iota(X) = n$, $\iota(Y) = n+1$, and $\varrho_{n+1} \le (1 + \varepsilon)  \mu^\ast(Y)$: By the claim above, it is enough to take $X := V_{1/n}$ and $Y = V_\alpha$ for some $\alpha \in \!\left[\max\!\left(\frac{1}{n+1},\frac{1}{(1+\varepsilon)n}\right)\!,\frac{1}{n}\right[$.
	
	So, to sum it up, $\iota$ also satisfies condition \ref{item:I6}, and the proof of the lemma is complete.
\end{proof}
Finally, we have what we need to confirm the independence of \ref{item:F2}: This is implied by the following theorem, which is straightforward by Lemmas \ref{lem:if_an_index_exists} and \ref{lem:existence_of_an_index} (we omit further details).
\begin{theorem}
	\label{th:independence_of_(F2)}
	Given $K \in [1,\infty]$, there exists a non-monotone, subadditive, $(-1)$-homogeneous, and translational invariant function $\theta^\ast: \PPc(\HHb) \to \RRb$ with the property that $\sup_{X  \in \PPc(\HHb)} \theta^\ast(X) = K$. In particular, there exists an upper quasi-density on $\HHb$ that is not an upper density.
\end{theorem}
\section{Examples}
\label{sec:examples}
We turn to examine a few examples that, on the one hand, illustrate how to build uncountably many upper densities and, on the other, generalize some of the most important instances of the notion of ``density'' that can be found in the literature. We begin with the so-called ``$\alpha$-densities''.
\begin{example}
	\label{exa:generalized_asymptotic_upper_densities}
	Fix $\alpha \in [-1,\infty[$ and $a,b \in \RRb$ such that $\sum_{i \in F_n} |i|^\alpha \to \infty$ as $n \to \infty$, where we set $F_n := [an,bn] \cap \HHb$ for all $n$. We denote by $\mathfrak F$ the sequence $(F_n)_{n \ge 1}$ and consider the function
	$$
	\dd^\ast(\mathfrak{F};\alpha): \PPc(\HHb) \to \RRb: X \mapsto
	\limsup_{n \to \infty} \frac{\sum_{i \in X  \cap  F_n} |i|^\alpha}{\sum_{i \in F_n} |i|^\alpha},
	$$
	with $\frac{0}{0} := 1$. It is seen (we omit details) that the dual of $\dd^\ast(\mathfrak{F};\alpha)$ is given by
	$$
	\dd_\ast(\mathfrak{F};\alpha): \PPc(\HHb) \to \RRb: X \mapsto \liminf_{n \to \infty} \frac{\sum_{i \in X  \cap  F_n} |i|^\alpha}{\sum_{i \in F_n} |i|^\alpha},
	$$
	and we have the following:
	\begin{proposition}
		\label{prop:generalized_upper_alpha-densities}
		The function $\dd^\ast(\mathfrak{F};\alpha)$ is an upper density.
	\end{proposition}
	\begin{proof}
		It is straightforward to check that $\dd^\ast(\mathfrak{F};\alpha)$ satisfies \ref{item:F1}, \ref{item:F2}, and \ref{item:F3}. As for \ref{item:F4b}, fix $X \subseteq \HHb$ and $h, k \in \NNb^+$. Given $\varepsilon \in {]0,1[}$, there exists $n_\varepsilon \in \NNb^+$ such that
		$
		0 < (1-\varepsilon)|ik+h|^\alpha \le |ik|^\alpha \le (1+\varepsilon)|ik+h|^\alpha
		$
		whenever $|i|\ge n_\varepsilon$. Together with the fact that $\dd^\ast(\mathfrak{F};\alpha)(S) = 0$ for every finite $S \subseteq \HHb$ (here we need that $\sum_{i\in F_n} |i|^\alpha = \infty$ as $n \to \infty$), this yields
		\begin{equation}
		\label{equ:translational_inequality_generalized_natural_updensity}
		\left\{
		\begin{array}{ll}
		\!\! \dd^\ast(\mathfrak{F};\alpha)(k \cdot X + h) = \dd^\ast(\mathfrak{F};\alpha)(k \cdot X_\varepsilon + h) \\
		\!\! (1-\varepsilon)k^\alpha u_\varepsilon \le \dd^\ast(\mathfrak{F};\alpha)(k \cdot X + h)
		\le (1+\varepsilon)k^\alpha u_\varepsilon
		\end{array}
		\right.\!\!,
		\end{equation}
		where, for ease of notation, we put
		\begin{equation}
		\label{equ:manipulations}
		u_\varepsilon := \limsup_{n \to \infty} \frac{\sum_{i \in X_\varepsilon  \cap  (k^{-1} \cdot  (F_n - h))} |i|^\alpha}{\sum_{i \in F_n} |i|^\alpha} =
		\limsup_{n \to \infty} \frac{\sum_{i \in X_\varepsilon  \cap  F_{\lfloor n/k \rfloor}} |i|^\alpha + \delta(n)}{\sum_{i \in F_n} |i|^\alpha},
		\end{equation}
		with
		$$
		X_\varepsilon := \{i \in X: |i|\ge n_\varepsilon\}
		\quad\text{and}\quad
		\delta(n) := \sum_{i \in X_\varepsilon  \cap  (k^{-1} \cdot  (F_n - h))} |i|^\alpha - \sum_{i \in X_\varepsilon  \cap  F_{\lfloor n/k \rfloor}} |i|^\alpha.
		$$
		Denote by $\Delta(n)$ the symmetric difference of $\HHb \cap (k^{-1} \cdot ( F_n-h))$ and $ F_{\lfloor n/k \rfloor}$. By the triangle inequality (we skip some details), we obtain
		\begin{equation*}
		0 \le \limsup_{n \to \infty} \frac{|\delta(n)|}{\sum_{i \in F_n} |i|^\alpha}
		\le \limsup_{n \to \infty} \frac{\sum_{i \in \Delta(n)} |i|^\alpha}{\sum_{i \in F_n} |i|^\alpha} = 0,
		\end{equation*}
		which, combined with \eqref{equ:manipulations}, implies
		\begin{equation*}
		\begin{split}
		u_\varepsilon
		& =
		\limsup_{n \to \infty} \frac{\sum_{i \in X_\varepsilon  \cap  F_{\lfloor n/k \rfloor}} |i|^\alpha}{\sum_{i \in F_n} |i|^\alpha}
		= \frac{1}{k^{\alpha+1}} \limsup_{n \to \infty} \frac{\sum_{i \in X_\varepsilon  \cap  F_{\lfloor n/k \rfloor}} |i|^\alpha}{\sum_{i \in F_{\lfloor n/k \rfloor}} |i|^\alpha}
		= \frac{1}{k^{\alpha+1}} \dd^\ast(\mathfrak{F};\alpha)(X)
		\end{split}
		\end{equation*}
		(note that $\dd^\ast(\mathfrak{F};\alpha)(X_\varepsilon) = \dd^\ast(\mathfrak{F};\alpha)(X)$, since $X_\varepsilon \subseteq X$ and $X \setminus X_\varepsilon$ is finite). So taking the limit of \eqref{equ:translational_inequality_generalized_natural_updensity} as $\varepsilon \to 0^+$, we get the desired conclusion.
	\end{proof}
	In continuity with \cite[Definition 1.4]{GiGr}, we call $\dd^\ast(\mathfrak{F};\alpha)$ and $\dd_\ast(\mathfrak{F};\alpha)$, respectively, the \textit{upper} and \textit{lower $\alpha$-density relative to $\mathfrak{F}$}.
	In particular, if $a=0$ and $b=1$, then $\dd^\ast(\mathfrak{F};-1)$ is the upper logarithmic density, cf. \cite[Chapter III.1, \S{ }1.2]{Tene}, and $\dd^\ast(\mathfrak{F};0)$ is the upper asymptotic density.
	
	Then it is not difficult to verify, see \cite[p. 37]{SZ}, that for every $\alpha \in {]-1,\infty[}$ the upper $\alpha$-density of the set $A := \bigcup_{n = 1}^\infty \llb 2^{2n}+1, 2^{2n+1} \rrb \subseteq \mathbf N^+$ is equal to $2^{\alpha+1}/(2^{\alpha+1} + 1)$. Clearly, this shows that the collection of all upper densities on $\mathbf H$ has at least the same cardinality of $\mathbf R$ (and hence is uncountable).
\end{example}
Another interesting example is offered by the upper Buck density (on $\mathbf H$):

\begin{definition}
	The upper Buck density \textup{(}on $\mathbf H$\textup{)} is the function
	\begin{equation}
	\label{equ:def-buck}
	\PPc(\HHb) \to \RRb: X \mapsto \inf_{S  \in \AAc: X  \subseteq  S} {\sf d}^\ast(S),
	\end{equation}
	where $\AAc$ denotes
	the collection of all sets that can be expressed as a finite union of arithmetic progressions of $\mathbf H$ \textup{(}that is, sets of the forms $k \cdot \mathbf H + h$ with $k \in \mathbf N^+$ and $h \in \mathbf H$, see \textup{\S{ }\ref{sec:notations}}\textup{)}.
\end{definition}
See, e.g., \cite{Bu0}, \cite[\S{ }7]{Niv} and \cite{SaTi} for the case $\mathbf H = \mathbf N^+$. This is generalized by the following:
\begin{example}
	\label{exa:buck's_measure_density}
	Let $\mu^\ast$ be an upper quasi-density on $\HHb$ and $\CCc$ a subfamily of $\PPc(\HHb)$ such that:
	\begin{enumerate}[label={\rm (\textsc{b}\arabic{*})}]
		\item\label{item:Buck_(B1)} $\HHb \in \CCc$;
		\item\label{item:Buck_(B2)} $X \cup Y \in \CCc$ for all $X,Y \in \CCc$;
		\item\label{item:Buck_(B3)} $k \cdot X + h \in \CCc$, for some $X \subseteq \HHb$ and $h, k \in \NNb^+$, if and only if $X \in \CCc$;
		\item\label{item:Buck_(B4b)} $X \cap (k \cdot \HHb + h) \in \CCc$ for all $X \in \CCc$ and $h,k \in \NNb^+$.
	\end{enumerate}
	In particular, it is seen that $\AAc$ satisfies \ref{item:Buck_(B1)}-\ref{item:Buck_(B4b)} if $\HHb = \ZZb$ or $\HHb = \NNb$, but not if $\HHb = \NNb^+$. On the other hand, it is not difficult to verify that conditions \ref{item:Buck_(B1)}-\ref{item:Buck_(B4b)} are all satisfied by taking
	\begin{equation}
	\label{equ:Ctheta}
	\CCc = \CCc_\theta := \{X \cup Y: X \in \AAc\text{ and } \theta(Y) = 0\} \subseteq \PPc(\HHb),
	\end{equation}
	provided $\theta$ is a function $\PPc(\HHb) \to \RRb$ such that:
	\begin{enumerate}[label={\rm (\roman{*})}]
		\item\label{item:buck's_families_theta_0} $\theta(X) = 0$ for every finite $X \subseteq \HHb$;
		\item\label{item:buck's_families_theta_i} $\theta(X \cup Y) \le \theta(X) + \theta(Y)$ for all $X,Y \subseteq \HHb$ (so $\theta$ is non-negative, cf. Proposition \ref{prop:elementary_properties_of_d-pairs}\ref{item:prop:elementary_properties_of_d-pairs(vib)});
		\item\label{item:buck's_families_theta_ii} $\theta(X) \le \theta(Y)$ whenever $X \subseteq Y \subseteq \HHb$;
		\item\label{item:buck's_families_theta_iii} $\theta(k \cdot X + h) = 0$, for some $X \subseteq \HHb$ and $h, k \in \NNb^+$, if and only if $\theta(X) = 0$.
	\end{enumerate}
	For instance, these conditions are fulfilled if $\theta$ is an upper density on $\HHb$, see Proposition \ref{prop:mu_of_finite_sets} below for \ref{item:buck's_families_theta_0}; the characteristic function of the infinite subsets of $\HHb$ (this is not an upper density), in which case $\CCc_\theta$ is the set of all subsets of $\PPc(\HHb)$ that can be represented as a finite union of arithmetic progressions of $\HHb$, or differ from these by finitely many integers; or the constant function $\PPc(\HHb) \to \RRb: X \mapsto 0$ (this is not an upper density either), in which case $\CCc_\theta = \PPc(\HHb)$.
	
	With this in mind, consider the function
	$$
	\bbf^\ast(\CCc; \mu^\ast): \PPc(\HHb) \to \RRb: X \mapsto \inf_{S  \in \CCc: X  \subseteq  S} \mu^\ast(S),
	$$
	which we denote by $\bbf^\ast$ whenever $\CCc$ and $\mu^\ast$ are clear from the context, and is well defined by the fact that $\HHb \in \CCc$ and $\mu^\ast(S) \in [0,1]$ for all $S \in \CCc$ (by Remark \ref{rem:upper_quasi-densities_are_non-negative}). We have:
	
	\begin{proposition}
		\label{prop:generalized_Buck's}
		The function $\bbf^\ast(\CCc; \mu^\ast)$ is an upper density.
	\end{proposition}
	\begin{proof}
		First, it is clear that $\bbf^\ast(\HHb) = 1$, because $\HHb \subseteq S$ for some $S \in \CCc$ only if $S = \HHb$, and on the other hand, $\HHb \in \CCc$ by \ref{item:Buck_(B1)} and $\mu^\ast(\HHb) = 1$ by the fact that $\mu^\ast$ satisfies \ref{item:F1}.
		
		Second, if $X \subseteq Y \subseteq \HHb$ and $Y \subseteq S \in \CCc$, then of course $X \subseteq S$. Therefore, we obtain that
		$$
		\bbf^\ast(X) = \inf_{S  \in \CCc: X  \subseteq  S} \mu^\ast(S) \le \inf_{S  \in \CCc: Y  \subseteq  S} \mu^\ast(S) = \bbf^\ast(Y).
		$$
		Third, if $X,Y \subseteq \HHb$ and $S,T \in \CCc$ are such that $X \subseteq S$ and $Y \subseteq T$, then $S \cup T \in \CCc$ by \ref{item:Buck_(B2)} and $X \cup Y \subseteq S \cup T$. This, together with the subadditivity of $\mu^\ast$, gives
		\begin{equation*}
		\begin{split}
		\bbf^\ast(X \cup Y)
		& \le \inf_{S,T  \in \CCc: X  \subseteq  S,  Y  \subseteq  T} \mu^\ast(S \cup T) \le \inf_{S,T  \in \CCc: X  \subseteq  S,  Y  \subseteq  T} (\mu^\ast(S) + \mu^\ast(T)) \\
		& = \inf_{S  \in \CCc: X  \subseteq  S} \mu^\ast(S) + \inf_{T  \in \CCc: Y  \subseteq  T} \mu^\ast(T) = \bbf^\ast(X) + \bbf^\ast(Y).
		\end{split}
		\end{equation*}
		Lastly, pick $X \subseteq \HHb$ and $h, k \in \NNb^+$. If $X \subseteq S \in \CCc$, then \ref{item:Buck_(B3)} yields $k \cdot X + h \subseteq k \cdot S + h \in \CCc$. Conversely, if $k \cdot X + h \subseteq T$ for some $T \in \CCc$, then $k \cdot X + h \subseteq T \cap (k \cdot \HHb + h) = k \cdot S + h$ for some $S \subseteq \HHb$, which implies, by \ref{item:Buck_(B3)} and \ref{item:Buck_(B4b)}, that $X \subseteq S \in \CCc$. Thus, we find
		$$
		\bbf^\ast(k \cdot X + h) = \inf_{T  \in \CCc: k  \cdot X + h \subseteq  T} \mu^\ast(T) = \inf_{S  \in \CCc: X \subseteq  S} \mu^\ast(k \cdot S + h) = \frac{1}{k} \bbf^\ast(X),
		$$
		where we have used that $\mu^\ast$ satisfies \ref{item:F4b}. It follows that $\bbf^\ast$ is an upper density.
	\end{proof}
	In view of Proposition \ref{prop:generalized_Buck's}, we will refer to $\bbf^\ast(\CCc; \mu^\ast)$ as the \textit{upper Buck density \textup{(}on $\HHb$\textup{)} relative to the pair $(\CCc,\mu^\ast)$}.
	If $\mu^\ast$ is the upper asymptotic density
	and $\CCc$ is the set $\CCc_\theta$ determined by \eqref{equ:Ctheta} when $\theta$ is the characteristic function of the infinite subsets of $\HHb$, then $\bbf^\ast(\CCc; \mu^\ast)$ is the upper Buck density, as given by \eqref{equ:def-buck}; in particular, Proposition \ref{prop:generalized_Buck's} generalizes \cite[Corollaries 2 and 3]{Pas}.
	
	It is perhaps interesting to note that the definition of $\bbf^\ast(\CCc; \mu^\ast)$ produces a ``smoothing effect'' on $\mu^\ast$, in the sense that $\bbf^\ast(\CCc; \mu^\ast)$ is monotone, no matter if $\mu^\ast$ is.
	
	In addition, we have the following result, which, together with Proposition \ref{prop:invariance_under_unions_with_finite_sets} in \S{ }\ref{sec:closing_remarks}, proves that the usual definition of the upper Buck density on $\NNb^+$ can be (slightly) simplified by establishing, as we do, that the \textit{upper Buck density on $\HHb$} is given by $\bbf^\ast(\mathscr{A}; \dd^\ast)$ if $\HHb = \ZZb$, and by the restriction to $\PPc(\HHb)$ of the upper Buck density on $\ZZb$ otherwise (recall that $\AAc$ does not satisfy \ref{item:Buck_(B1)}-\ref{item:Buck_(B4b)} if $\HHb = \NNb^+$, and that $\dd^\ast$ accounts only for the positive part of a subset of $\HHb$).
	
	\begin{proposition}
		Let $\CCc^\sharp := \{X \cup \mathcal{H}: X \in \CCc, \mathcal{H} \subseteq \HHb, |\mathcal{H}| < \infty\}$. Then 
		$$
		\bbf^\ast(\CCc; \mu^\ast) = \bbf^\ast(\CCc^\sharp; \mu^\ast).
		$$
	\end{proposition}
	\begin{proof}
		Pick $X \subseteq \HHb$. Since $\CCc \subseteq \CCc^\sharp$, it is immediate that $\bbf^\ast(\CCc^\sharp; \mu^\ast)(X) \le \bbf^\ast(\CCc; \mu^\ast)(X)$, where we use, as in the proof of Proposition \ref{prop:generalized_Buck's}, that if $\emptyset \neq A \subseteq B \subseteq \mathbf R$, then $\inf(B) \le \inf(A)$.
		It remains to prove that $\bbf^\ast(\CCc; \mu^\ast)(X) \le \bbf^\ast(\CCc^\sharp; \mu^\ast)(X)$.
		
		For this, fix a real $\varepsilon > 0$. By definition, there exists $T \in \CCc^\sharp$ for which $X \subseteq T$ and $\mu^\ast(T) \le \bbf^\ast(\CCc^\sharp; \mu^\ast)(X) + \frac{\varepsilon}{2}$. On the other hand, $T \in \mathscr{C}^\sharp$ if and only if $T = Y \cup \mathcal{H}$ for some $Y \in \CCc$ and $\mathcal{H} \subseteq \HHb$ with $|\mathcal{H}| < \infty$.
		So, set $\mathcal{V} := \bigcup_{h \in \mathcal{H}}(k \cdot \HHb + h)$, where $k$ is an integer $\ge \frac{2}{\varepsilon}|\mathcal{H}|$, and notice that $(Y + k) \cup \mathcal{V} = (T + k) \cup \mathcal{V}$.
		
		It follows from the above and conditions \ref{item:Buck_(B1)}-\ref{item:Buck_(B3)} that $X+k \subseteq (T + k) \cup \mathcal{V} \in \CCc$,
		and this in turn implies, by the fact that $\bbf^\ast(\CCc; \mu^\ast)(X)$ is translational invariant (by Proposition \ref{prop:generalized_Buck's}) and $\mu^\ast$ is translational invariant and subadditive (by hypothesis), that
		\begin{equation*}
		\begin{split}
		\bbf^\ast(\CCc; \mu^\ast)(X)
		& = \bbf^\ast(\CCc; \mu^\ast)(X+k) \le \mu^\ast(T+k) + \mu^\ast(\mathcal{V}) \\
		& = \mu^\ast(T) + \frac{|\mathcal{H}|}{k} \le \bbf^\ast(\CCc^\sharp; \mu^\ast)(X) + \varepsilon.
		\end{split}
		\end{equation*}
		This is enough to complete the proof, since $\varepsilon$ was arbitrary.
	\end{proof}
	For future reference, we take the Buck density on $\HHb$ to be the density induced by $\bbf^\ast(\mathscr{A}; \dd^\ast)$, and we call the dual of $\bbf^\ast(\mathscr{A}; \dd^\ast)$ the lower Buck density on $\HHb$.
	
	We are left with the question of providing a ``convenient expression'' for the lower dual of $\bbf^\ast(\CCc; \mu^\ast)$, here denoted by $\bbf_\ast(\CCc; \mu^\ast)$:
	This seems unfeasible in general. However, if we assume
	\begin{enumerate}[resume, label={\rm (\textsc{b}\arabic{*})}]
		\item\label{item:Buck_(B4)} $X^c \in \CCc$ whenever $X \in \CCc$ (namely, $\CCc$ is closed under complementation),
	\end{enumerate}
	then it is not difficult to verify that $\bbf_\ast(\CCc; \mu^\ast)(X)$ is given by the function
	$$
	\PPc(\HHb) \to \RRb: X \mapsto \sup_{T  \in \CCc: T  \subseteq  X} \mu_\ast(T),
	$$
	with $\mu_\ast$ being the lower dual of $\mu^\ast$. In fact, if $\CCc$ satisfies \ref{item:Buck_(B4)} then,  for all $X \subseteq \HHb$,
	\begin{equation*}
	\begin{split}
	\bbf_\ast(\CCc; \mu^\ast)(X)
	& := 1 - \bbf^\ast(\CCc; \mu^\ast)(X^c) = 1 - \inf_{S \in \CCc: X^c  \subseteq  S} \mu^\ast(S) = 1 - \inf_{S  \in \CCc: S^c  \subseteq  X} \mu^\ast(S) \\
	& \phantom{:}= 1 - \inf_{T  \in \CCc: T  \subseteq  X} \mu^\ast(T^c) =
	\sup_{T  \in \CCc: T  \subseteq  X} (1 - \mu^\ast(T^c)) = \sup_{T  \in \CCc: T  \subseteq  X} \mu_\ast(T).
	\end{split}
	\end{equation*}
	Lastly, we note that the above construction is not vacuous, in the sense that, for some choice of $\mu^\ast$ and $\CCc$, we have $\mu^\ast = \bbf^\ast(\PPc(\HHb); \mu^\ast) \ne \bbf^\ast(\CCc; \mu^\ast) \ne \bbf^\ast(\AAc; \mu^\ast)$.
	Indeed, let $\mathsf{ld}^\ast$ denote the upper logarithmic density on $\HHb$ (see Example \ref{exa:generalized_asymptotic_upper_densities}). By Proposition \ref{prop:modular_criterion} (we omit details), there exists $X \subseteq \NNb^+$ with $\mathsf{ld}^\ast(X) = \dd^\ast(X) = 0$, but $\bbf^\ast(\AAc; \dd^\ast)(X) = 1$. 
	On the other hand, we get, e.g., from \cite[Theorem 3]{LuPo} that there is $Y \subseteq \NNb^+$ such that $\mathsf{ld}^\ast(Y) = 0$ and $\dd^\ast(Y) = 1$.
	So $\dd^\ast \ne \bbf^\ast(\CCc; \dd^\ast) \ne \bbf^\ast(\AAc; \dd^\ast)$, where, consistently with \eqref{equ:Ctheta}, we set $\CCc := \{S \cup T: S \in \AAc \text{ and } \mathsf{ld}^\ast(T) = 0\}$.
\end{example}
Our last example is about another classic of the ``literature on densities'':
\begin{example}
	Denote by $\zeta$ the function ${]1,\infty[} \to \RRb: s \mapsto \sum_{n = 1}^\infty n^{-s}$, namely, the restriction of the Riemann zeta to the interval $]1,\infty[$.
	Then consider the function
	$$
	\mathfrak{a}^\star: \mathcal{P}(\mathbf{H}) \to \RRb: X \mapsto \limsup_{s \to 1^+} \frac{1}{ \zeta(s)} \sum_{i  \in X^+} \frac{1}{i^s}.
	$$
	We claim that $\mathfrak{a}^\ast$ is an upper density, which we refer to as the \textit{upper analytic density} (on $\HHb$) for consistency with \cite[Part III, \S{ }1.3]{Tene}, where the focus is on the case $\HHb = \NNb^+$.
	
	In fact, it is straightforward to check that $\mathfrak{a}^\ast$ satisfies \ref{item:F1}-\ref{item:F4}. As for \ref{item:F5}, fix $X \subseteq \HHb$ and $h \in \NNb$, and pick $\varepsilon \in {]0,1[}$. There exists $n_\varepsilon \in \NNb^+$ such that
	\begin{equation}
	\label{equ:inequality_for_large_arguments}
	0 < (1-\varepsilon)|i+h| \le |i| \le (1+\varepsilon)|i+h|
	\end{equation}
	for $|i|\ge n_\varepsilon$. Set $X_\varepsilon := \{i \in X: i \ge n_\varepsilon\}$. Then $\mathfrak{a}^\ast(S) = 0$, and hence $\mathfrak{a}^\ast(T) = \mathfrak{a}^\ast(S \cup T)$, for all $S,T \subseteq \HHb$ with $|S| < \infty$. Thus $\mathfrak{a}^\star(X)= \mathfrak{a}^\star(X_\varepsilon)$, and by \eqref{equ:inequality_for_large_arguments} we have
	$$
	\limsup_{s \to 1^+} \frac{1}{(1+\varepsilon)^s  \zeta(s)} \sum_{i  \in X_\varepsilon} \frac{1}{(i+h)^s} \le \mathfrak{a}^\star(X) \le \limsup_{s \to 1^+} \frac{1}{(1-\varepsilon)^s  \zeta(s)} \sum_{i  \in X_\varepsilon} \frac{1}{(i+h)^s}.
	$$
	This, in the limit as $\varepsilon \to 0^+$, yields $\mathfrak{a}^\ast(X) = \mathfrak{a}^\ast(X_\varepsilon + h) = \mathfrak{a}^\ast(X + h)$, where we have used again that $\mathfrak{a}^\ast$ is ``invariant under union with finite sets''.
	
	Therefore, $\mathfrak{a}^\ast$ is an upper density, whose lower dual (we omit details) is given by
	$$
	\mathfrak{a}_\ast: \PPc(\HHb) \to \RRb: X \mapsto
	\liminf_{s \to 1^+} \frac{1}{ \zeta(s)} \sum_{i  \in X^+} \frac{1}{i^s}.
	$$
\end{example}
More examples will come later, when we provide some simple criteria to construct new upper densities from old ones
(see, e.g., the discussion at the beginning of \S{ }\ref{sec:others} and Proposition \ref{prop:convex_combinations}).
\section{Range of upper and lower densities}
\label{sec:range}
There are a number of natural questions that may be asked about upper densities. In the light of Remark \ref{rem:upper_quasi-densities_are_non-negative}, one of the most basic of them is probably the following:
\begin{question}
	\label{quest:A}
	Let $\mu^\ast$ be an upper density on $\HHb$. Is it true that $\imag(\mu^\ast) = [0,1]$?
\end{question}
We will actually prove (Theorem \ref{th:image_of_upper_densities}) that the image of every quasi-density is the whole interval $[0,1]$: Therefore, so will be the image of every upper and lower quasi-density. 
This generalizes \cite[Theorem 6]{Bu0} and \cite[Theorem 5]{Pas} (the case of the upper Buck density on $\NNb^+$), and similar results that are known for other classical densities (cf. also Question \ref{open:density_sets} and \cite[\S{ }3]{Mah}). Moreover, it is kind of an analogue for upper quasi-densities of
a theorem of A.~A.~Liapounoff \cite[Theorem 1]{Lia} on the convexity of the range of a non-atomic countably additive vector measure (with values in $\RRb^n$).
\begin{proposition}
	\label{prop:mu_of_finite_sets}
	Let $\mu^\ast$ be a function $\PPc(\HHb) \to \RRb$ that satisfies axioms \ref{item:F1}, \ref{item:F3} and \ref{item:F4b}. If $X$ is a finite subset of $\HHb$, then $\mu^\ast(X) = 0$.
\end{proposition}
\begin{proof}
	Let $X$ be a finite subset of $\HHb$. We have by Proposition \ref{prop:elementary_properties_of_d-pairs}\ref{item:prop:elementary_properties_of_d-pairs(iiib)} that $\mu^\ast(\emptyset) = 0$, and by Remark \ref{remark:sigma-subadditivity} that $\mu^\ast(\{k\}) = 0$ for all $k \in \mathbf H$. So, we obtain from the above and parts \ref{item:prop:elementary_properties_of_d-pairs(iib)} and \ref{item:prop:elementary_properties_of_d-pairs(vib)} of Proposition \ref{prop:elementary_properties_of_d-pairs} that $0 \le \mu^\ast(X) \le  \sum_{x \in X} \mu^\ast(\{x\}) = 0$, which completes the proof.
\end{proof}
Incidentally, observe that, in view of Example \ref{exa:minimum_function}, conditions \ref{item:F1}-\ref{item:F4} alone are not sufficient for Proposition \ref{prop:mu_of_finite_sets} to hold, while the result is independent of \ref{item:F2}.
\begin{proposition}
	\label{prop:unions_of_translates}
	Let $\mu^\ast$ be an upper quasi-density, and for a fixed $k \in \NNb^+$ let $h_1, \ldots, h_n \in \NNb$ be such that $h_i \not\equiv h_j \bmod k$ for $1 \le i < j \le n$ and set $X := {\bigcup}_{i=1}^n (k \cdot \HHb + h_i)$. Then, for every finite $\mathcal{V} \subseteq \HHb$ we have 
	$$
	\mu^\ast(X \cup \mathcal{V}) = \mu^\ast(X \setminus \mathcal{V}) = \mu_\ast(X \cup \mathcal{V}) = \mu_\ast(X \setminus \mathcal{V}) = \frac{n}{k},
	$$
	where $\mu_\ast$ is the lower dual of $\mu^\ast$.
\end{proposition}
\begin{proof}
	Let $l_i$ be, for each $i=1, \ldots, n$, the remainder of the integer division of $h_i$ by $k$ (in such a way that $0 \le l_i < k$), and set 
	$$
	Y := {\bigcup}_{l \in \mathcal{H}} (k \cdot \HHb + l), 
	\quad\text{where }\mathcal{H} := \llb 0, k-1 \rrb \setminus \{l_1, \ldots, l_n\}.
	$$
	Clearly, $\HHb = X \cup Y \cup S$ for some finite $S \subseteq \HHb$. Therefore,
	we have from axioms \ref{item:F1}, \ref{item:F3}, and \ref{item:F4b} and Propositions \ref{prop:mu_of_finite_sets} and \ref{prop:elementary_properties_of_d-pairs}\ref{item:prop:elementary_properties_of_d-pairs(iib)} that, however we choose a finite $\mathcal{V} \subseteq \HHb$,
	\begin{equation*}
	\begin{split}
	1 = \mu^\ast(\HHb)
	& \le \mu^\ast(X \cup \mathcal{V}) + \mu^\ast(Y \cup S) \le \mu^\ast(X) + \mu^\ast(Y) + \mu^\ast(S) + \mu^\ast(\mathcal{V}) \\
	& = \mu^\ast(X) + \mu^\ast(Y) \le n \mu^\ast(k \cdot \HHb) + (k-n) \mu^\ast(k \cdot \HHb) = k \mu^\ast(k \cdot \HHb) = 1,
	\end{split}
	\end{equation*}
	which is possible only if $\mu^\ast(X \cup \mathcal{V}) = \mu^\ast(X) = n\mu^\ast(k \cdot \HHb) = \frac{n}{k}$.
	
	On the other hand, if $\mathcal{V} \subseteq \HHb$ is finite, then for each $i \in \llb 1, n \rrb$ there exists a set $S_i \subseteq \mathbf H$ such that $(k \cdot \mathbf H + h_i) \setminus \mathcal V = k \cdot S_i + h_i$. Thus, we conclude from the above, Proposition \ref{prop:mu_of_finite_sets}, and the properties of $\mu^\ast$ (we skip some details) that
	\begin{equation*}
	\begin{split}
	\frac{n}{k} & = \mu^\ast(X) \le \mu^\ast(\mathcal V) + \sum_{i=1}^n \mu^\ast((k \cdot \HHb + h_i) \setminus \mathcal V) \\
	& =
	\sum_{i=1}^n \mu^\ast(k \cdot S_i + h_i) = \frac{1}{k} \sum_{i=1}^n \mu^\ast(S_i) \le \frac{n}{k},
	\end{split}
	\end{equation*}
	which shows that $\mu^\ast(X \setminus \mathcal V) = \mu^\ast(X) = \frac{n}{k}$.
	
	As for $\mu_\ast$, it is straightforward that if $\mathcal V$ is a finite subset of $\HHb$ then $|Y \triangle (X \cup \mathcal V)^c| < \infty$ and $|Y \triangle (X \setminus \mathcal V)^c| < \infty$,
	which, together with the first part, yields 
	$$
	\mu^\ast((X \cup \mathcal V)^c) = \mu^\ast((X \setminus \mathcal V)^c) = \mu^\ast(Y) = 1 - \frac{n}{k}, 
	$$
	and hence $\mu_\ast(X \cup \mathcal V) = \mu_\ast(X \setminus \mathcal V) = \frac{n}{k}$.
\end{proof}
Proposition \ref{prop:unions_of_translates} can be regarded, in the light of Proposition \ref{prop:mu_of_finite_sets}, as a supplement to Proposition \ref{prop:null_invariance}\ref{prop:null_invariance(ii)}, and it is already enough to imply the following corollary (we omit further details), which falls short of an answer to Question \ref{quest:A}, but will be used later in the proof of Theorem \ref{th:image_of_upper_densities}.
\begin{corollary}
	\label{cor:density_of_image}
	Let $\mu$ be an upper quasi-density on $\HHb$. Then $\mathbf Q \cap [0,1] \subseteq \imag(\mu)$.
\end{corollary}
The next result is essentially an extension of Proposition \ref{prop:unions_of_translates}.

\begin{proposition}
	\label{prop:union_of_disjoint_classes_different_modules}
	Let $\mu^\ast$ be an upper quasi-density, and assume that $(k_i)_{i \ge 1}$ and $(h_i)_{i \ge 1}$ are integer sequences with the property that:
	\begin{enumerate}[label={\rm (\roman{*})}]
		\item\label{item:prop:union_of_disjoint_classes_different_modules(i)} $k_i \ge 1$ and $k_i \mid k_{i+1}$ for each $i \in \NNb^+$;
		\item\label{item:prop:union_of_disjoint_classes_different_modules(ii)} Given $i,j \in \NNb^+$ with $i < j$, there exists no $x \in \ZZb$ such that $k_i x + h_i \equiv h_j \bmod k_j$.
	\end{enumerate}
	Then $\mu^\ast(X_n) = \sum_{i=1}^n \frac{1}{k_i}$, where $X_n := \bigcup_{i=1}^n (k_i \cdot \HHb + h_i)$.
\end{proposition}
\begin{proof}
	Fix $n \in \NNb^+$. By condition \ref{item:prop:union_of_disjoint_classes_different_modules(i)}, we have, for each $i \in \llb 1, n \rrb$,
	$$k_i \cdot \HHb + h_i = \bigcup_{l = 0}^{k_i^{-1}k_n-1} (k_n \cdot \HHb  + k_i l + h_i).$$
	Therefore, we find that
	\begin{equation}
	\label{equ:decomposition_of_set_into_union_of_disjoint_residue_classes_mod_kn}
	X_n = \bigcup_{i=1}^n \bigcup_{l = 0}^{k_i^{-1}k_n-1} (k_n \cdot \HHb  + k_i l + h_i).
	\end{equation}
	On the other hand, if $1 \le i < j \le n$ then $k_i l_i + h_i \not\equiv k_j l_j + h_j \bmod k_n$ for all $l_i \in \llb 0, k_i^{-1} k_n - 1 \rrb$ and $l_j \in \llb 0, k_j^{-1} k_n - 1 \rrb$, otherwise we would have that $l_i$ is an integer solution to the congruence $k_i x + h_i \equiv h_j \bmod k_j$ (by the fact that $k_j \mid k_n$), in contradiction to condition \ref{item:prop:union_of_disjoint_classes_different_modules(ii)}.
	Thus, it follows by \eqref{equ:decomposition_of_set_into_union_of_disjoint_residue_classes_mod_kn} and Proposition \ref{prop:unions_of_translates} that 
	\begin{equation*}
	\mu^\ast(X_n) = \frac{1}{k_n} \sum_{i=1}^n k_i^{-1} k_n = \sum_{i=1}^n \frac{1}{k_i},
	\end{equation*}
	which concludes the proof.
\end{proof}
While \ref{item:F2} is, by Theorem \ref{th:independence_of_(F2)}, independent of \ref{item:F1}, \ref{item:F3} and \ref{item:F4b}, the latter conditions are almost sufficient to prove a weak form of \ref{item:F2}, as shown in the next two statements,
where $\AAc^\sharp$ denotes the set of all subsets of $\HHb$ that are finite unions
of arithmetic progressions of $\HHb$, or differ from these by a finite number of integers (in particular, $\emptyset \in \AAc^\sharp$), cf. Example \ref{exa:buck's_measure_density}.
\begin{proposition}
	\label{prop:relaxed_(F2)_with_upper_arithmetic_bound}
	Let $\mu^\ast$ be an upper quasi-density on $\HHb$, and pick $X \in \PPc(\HHb)$ and $Y \in \mathscr{A}^\sharp$ such that $X \subseteq Y$. Then $\mu^\ast(X) \le \mu^\ast(Y)$.
\end{proposition}
\begin{proof}
	Since $Y \in \mathscr{A}^\sharp$, there exist $k \in \NNb^+$ and $\HHc \subseteq \llb 0, k-1 \rrb$ such that the symmetric difference of $Y$ and $\bigcup_{h \in \HHc} (k \cdot \HHb + h)$ is finite.
	Because $X \subseteq Y$, it follows that the relative complement of $\bigcup_{h \in \HHc} X_h$ in $X$, where $X_h := X \cap (k \cdot \HHb + h) \subseteq X$, is finite too. Therefore, we get by Propositions \ref{prop:mu_of_finite_sets} and \ref{prop:elementary_properties_of_d-pairs}\ref{item:prop:elementary_properties_of_d-pairs(iib)} that $\mu^\ast(X) \le \sum_{h \in \HHc}\mu^\ast(X_h)$, and by Proposition \ref{prop:unions_of_translates} that $\mu^\ast(Y) = \frac{1}{k}|\HHc|$.
	
	On the other hand, we have that, however we choose $h \in \HHc$, there is a set $S_h \subseteq \HHb$ for which $X_h = k \cdot S_h + h$. Hence, we infer from the above, \ref{item:F4b}, and Remark \ref{rem:upper_quasi-densities_are_non-negative} that
	\begin{equation*}
	\mu^\ast(X) \le \sum_{h \in \HHc} \mu^\ast(X_h) = \sum_{h \in \HHc} \mu^\ast(k \cdot S_h + h) = \frac{1}{k} \sum_{h \in \HHc} \mu^\ast(S_h) \le \mu^\ast(Y),
	\end{equation*}
	which finishes the proof.
\end{proof}
\begin{corollary}
	\label{cor:inequality_for_density}
	Let $X \subseteq \HHb$ and $Y, Z \in \mathscr{A}^\sharp$ such that $Y \subseteq X \subseteq Z$, and assume $\mu^\ast$ is an upper quasi-density on $\HHb$. Then $Y, Z \in \dom(\mu)$ and $\mu(Y) \le \mu_\ast(X) \le \mu^\ast(X) \le \mu(Z)$, where $\mu$ and $\mu_\ast$ are, respectively, the quasi-density induced by and the lower dual of $\mu^\ast$.
\end{corollary}
\begin{proof}
	First, $Y \subseteq X$ implies $X^c \subseteq Y^c$, and it is clear that $Y^c \in \AAc^\sharp$. So we get from Propositions \ref{prop:elementary_properties_of_d-pairs}\ref{item:prop:elementary_properties_of_d-pairs(vib)} and \ref{prop:relaxed_(F2)_with_upper_arithmetic_bound} that $\mu_\ast(X) \le \mu^\ast(X) \le \mu^\ast(Z)$ and $\mu^\ast(X^c) \le \mu^\ast(Y^c)$, and the latter inequality gives $\mu_\ast(Y) \le \mu_\ast(X)$. This is enough to complete the proof, since we know from Proposition \ref{prop:unions_of_translates} that $Y, Z \in \dom(\mu)$, and therefore $\mu_\ast(Y) = \mu(Y)$ and $\mu^\ast(Z) = \mu(Z)$.
\end{proof}
Finally, we are ready to answer Question \ref{quest:A}.
\begin{theorem}
	\label{th:image_of_upper_densities}
	Let $\mu$ be the quasi-density induced by an upper quasi-density $\mu^\ast$ on $\HHb$. Then the range of $\mu$ is $[0,1]$. In particular, $\imag(\mu^\ast) = \imag(\mu_\ast) = [0,1]$, where $\mu_\ast$ is the lower dual of $\mu^\ast$.
\end{theorem}
\begin{proof}
	Remark \ref{rem:upper_quasi-densities_are_non-negative} and Corollary \ref{cor:density_of_image} yield 
	$\mathbf Q \cap [0,1] \subseteq \imag(\mu) \subseteq [0,1]$.
	So, fix an irrational number $\alpha \in [0,1]$. Then, there is uniquely determined an increasing sequence $(a_i)_{i \ge 1}$ of positive integers with $\alpha = \sum_{i = 1}^\infty 2^{-a_i}$. Accordingly, let $X_i$ denote, for each $i \in \NNb^+$, the set $
	X_i := 2^{a_i}\cdot\HHb + r_i$,
	where $r_i := \sum_{j=1}^{i-1} 2^{a_j-1}$. Lastly, define $X := \bigcup_{i = 1}^\infty X_i$.
	
	Given $n \in \NNb^+$, we note that $X_i \subseteq 2^{a_n-1} \cdot \HHb + r_n$ for every $i \ge n$, because $x \in X_i$ if and only there exists $y \in \HHb$ for which $x = 2^{a_i}y + \sum_{j=1}^{i-1} 2^{a_j - 1}$, so that
	$x = 2^{a_n-1} z + r_n$  for some $z \in \HHb$. Taking $Y_n := \bigcup_{i=1}^n X_i$, we obtain
	$$
	Y_n \subseteq X \subseteq Y_n \cup (2^{a_n-1} \cdot \HHb + r_n),
	$$
	which, in turn, implies by Corollary \ref{cor:inequality_for_density} and
	axiom \ref{item:F3} that
	\begin{equation}
	\label{equ:squeeze}
	\begin{split}
	\mu(Y_n) & \le \mu_\ast(X) \le \mu^\ast(X) \le \mu(Y_n \cup (2^{a_n-1} \cdot \HHb + r_n)) \\
	& \le \mu(Y_n) + \mu(2^{a_n-1} \cdot \HHb + r_n).
	\end{split}
	\end{equation}
	On the other hand, it is seen that, however we choose $i,j \in \NNb^+$ with $i < j$, there exists no $x \in \ZZb$ such that $2^{a_i} x + r_i \equiv r_j \bmod 2^{a_j}$: Otherwise, 
	$
	2^{a_i} x \equiv \sum_{l = i}^{j-1} 2^{a_l-1} \bmod 2^{a_j},
	$
	that is, $2x \equiv \sum_{l = i}^{j-1} 2^{a_l-a_i} \bmod 2^{a_j - a_i + 1}$, which is impossible, because $\sum_{l = i}^{j-1} 2^{a_l-a_i}$ is an odd integer and $a_j - a_i + 1 > 0$.
	It follows from Proposition \ref{prop:union_of_disjoint_classes_different_modules}, equation
	\eqref{equ:squeeze}, and axiom \ref{item:F4b} that
	$$
	\sum_{i=1}^n \frac{1}{2^{a_i}} = \mu(Y_n) \le \mu_\ast(X) \le \mu^\ast(X) \le \frac{1}{2^{a_n-1}} + \sum_{i=1}^n \frac{1}{2^{a_i}}.
	$$
	So, passing to the limit as $n \to \infty$, we get that $\mu_\ast(X) = \mu^\ast(X) = \alpha$. Thus, $X \in \dom(\mu)$ and $\mu(X) = \alpha$, which completes the proof, since $\alpha$ was arbitrary.
\end{proof}
Incidentally, it has been recently proved in \cite[Theorem 1]{LT17} that upper and lower quasi-densities have a kind of intermediate value property, which is actually much stronger than the ``In particular'' part of Theorem \ref{th:image_of_upper_densities} (cf. also Question \ref{open:Darboux_for_densities} below).
\section{Structural results}
\label{sec:others}
Let $\GGc$ be a subset of $\FFc := \{\ref{item:F1}, \ldots, \ref{item:F5}\}$, where axioms $\ref{item:F1}, \ldots, \ref{item:F5}$ are viewed as words of a suitable formal language; in particular, we write $\FFc_1$ for $\FFc \setminus \{\ref{item:F1}\}$, $\FFc_2$ for $\FFc \setminus \{\ref{item:F2}\}$, and so forth.
We denote by $\MMc^\ast(\GGc)$ the set of all functions $\mu^\ast: \PPc(\HHb) \to \mathbf{R}$ that satisfy the axioms in $\GGc$ and the condition $\imag(\mu^\ast) \subseteq [0,1]$, and by $\MMc_\ast(\GGc)$ the set of the duals of the functions in $\MMc^\ast(\GGc)$.
In particular, $\MMc^\ast(\FFc)$ and $\MMc^\ast(\FFc_2)$ are, respectively, the set of all upper densities and the set of all upper quasi-densities on $\HHb$. We may ask the following (vague) question:
\begin{question}
	Is there anything interesting about the ``structure'' of $\MMc^\ast(\GGc)$ and $\MMc_\ast(\GGc)$?
\end{question}
For a partial answer, we regard $\MMc^\ast(\GGc)$ and $\MMc_\ast(\GGc)$
as subsets of $\mathcal{B}(\PPc(\HHb),\RRb)$, the real vector space of all bounded functions $f: \PPc(\HHb) \to \RRb$,
endowed with the (partial) order ${\preceq}$ given by $f \preceq g$ if and only if $f(X) \le g(X)$ for all $X\subseteq \HHb$ (as usual, we will write $f \prec g$ if $f \preceq g$ and $f \ne g$).

Given $q \in \RRb^+$, we say that a subset $\mathcal{F}$ of non-negative and uniformly bounded functions of $\mathcal{B}(\PPc(\HHb),\RRb)$ is \textit{countably $q$-convex} if, however we choose a $[0,1]$-valued sequence $(\alpha_n)_{n \ge 1}$ such that $\sum_{n = 1}^\infty \alpha_n = 1$ and an $\mathcal{F}$-valued sequence $(f_n)_{n \ge 1}$, the function
$$
f: \PPc(\HHb) \to \RRb: X \mapsto {\left(\sum_{n = 1}^\infty \alpha_n (f_n(X))^{q}\right)}^{1/q}
$$
is still in $\mathcal F$ (notice that $f$ is well defined, thanks to the assumption that there exists $K \in \mathbf R^+$ such that $f_n(X) \le K$ for every $X \in \mathcal P(\mathbf H)$ and $n \in \mathbf N^+$).
In particular, ``$\mathcal{F}$ being sigma-convex'' is the same as ``$\mathcal{F}$ being countably $1$-convex'' (which in turn implies that $\mathcal F$ is convex).
\begin{proposition}
	\label{prop:convex_combinations}
	Let $q \in [1,\infty[$. Then $\MMc^\ast(\GGc)$ is a countably $q$-convex set.
\end{proposition}
\begin{proof}
	Let $\mu^\ast := \sum_{n=1}^\infty \alpha_n \mu_n^{q}$, where $(\alpha_n)_{n \ge 1}$ is a sequence of non-negative real numbers such that $\sum_{n = 1}^\infty \alpha_n = 1$, and $(\mu_n)_{n \ge 1}$ is an $\MMc^\ast(\GGc)$-valued sequence. It is straightforward that $\mu^\ast \in \MMc^\ast(\GGc)$: In particular, it is an easy consequence of Minkowski's inequality for sums that, if every $\mu_n$ is subadditive, then so is $\mu^\ast$.
\end{proof}
The above proposition implies (we omit details) that $\MMc^\ast(\FFc)$ has at least the same cardinality of $\RRb$ (cf. Example \ref{exa:generalized_asymptotic_upper_densities}), 
and it leads to the following corollary, whose simple proof we leave as an exercise for the reader (later on, we will often do the same with no further comment).
\begin{corollary}
	\label{cor:convexity}
	Both $\MMc^\ast(\GGc)$ and $\MMc_\ast(\GGc)$ are sigma-convex sets.
\end{corollary}
On the other hand, we have from Propositions \ref{prop:null_invariance}\ref{prop:null_invariance(iv)}, \ref{prop:elementary_properties_of_d-pairs}, \ref{prop:mu_of_finite_sets}, and \ref{prop:invariance_under_unions_with_finite_sets} that if
$\mu^\ast$ is an upper density with lower dual $\mu_\ast$, then $\mu^\ast$ and $\mu_\ast$
satisfy axioms \ref{item:FS2}-\ref{item:FS9} with $\HHb = \NNb^+$, $\delta_\ast = \mu_\ast$, and $\delta^\ast = \mu^\ast$.
Thus, it is natural to ask if a lower density in the sense of our definitions must be also a lower density in the sense of Freedman and Sember's work \cite{FrSe}: The next example shows that this is not the case, and we have already noted in Example \ref{exa:Freedman-Sember_densities_are_not_scalable} that the converse does not hold either.
\begin{example}
	\label{exa:freedman&sember}
	Let $f$ and $g$ be two upper densities, and let $\alpha \in [0,1]$ and $q \in [1,\infty[$. We have by Proposition \ref{prop:convex_combinations} that the function 
	$$
	h^\ast := (\alpha f^q + (1-\alpha) g^q)^{\frac{1}{q}}
	$$
	is an upper density too.
	In particular, assume from now on that $f$ is the upper asymptotic density
	and $g$ the upper Banach density. Accordingly, fix $a \in {]0,1]}$ and define
	$$
	V_a := \bigcup_{n = 1}^\infty \llb a{}(2n-1)!+(1-a) (2n)!, (2n)!\rrb.
	$$ 
	Then, consider the sets
	$$
	X := V_a \cup (V_a^{ c} \cap (2\cdot \HHb))
	\quad\text{and}\quad
	Y := V_a \cup (V_a^{ c} \cap (2 \cdot \HHb+1)) \cup \{1\}.
	$$
	It is clear that $X \cup Y = \HHb$ and $X \cap Y = V_a$, and we get from Lemma \ref{lem:asymptotic_upper_density_of_special_sets} that $f(V_a) = a$. So, $f$ being an upper density (in the sense of our definitions) yields
	$$
	f(X) \le f(V_a) + f(V_a^{ c} \cap (2 \cdot \HHb)) \le a +\frac{1}{2},
	$$
	and similarly $f(Y) \le a + \frac{1}{2}$. On the other hand, $V_a$ contains arbitrarily large intervals of consecutive integers, hence $g(X) = g(Y) = g(V_a) = 1$. It follows
	$$
	\left\{
	\begin{array}{l}
	\!\! 1+h^\ast(X \cap Y) = 1 + (\alpha +a^q(1-\alpha))^{\frac{1}{q}} \\
	\!\! h^\ast(X) + h^\ast(Y) \le 2 \cdot (\alpha + (a+1/2)^q(1-\alpha))^{\frac{1}{q}}
	\end{array}
	\right.\!\!.
	$$
	With this in hand, suppose for a contradiction that $1+h^\ast(A \cup B) \le h^\ast(A) + h^\ast(B)$ for all $A, B \subseteq \HHb$ such that $A \cup B = \HHb$ (regardless of the actual values of the parameters $a$, $\alpha$ and $q$), which is equivalent to saying that the conjugate of $h^\ast$ satisfies \ref{item:FS1}. Then, we have from the above that
	$$
	1 + \alpha^{\frac{1}{q}} \le 2 \cdot (\alpha +(a+1/2)^q(1-\alpha))^{\frac{1}{q}},
	$$
	which implies, in the limit as $a \to 0^+$, that
	$
	1 + \alpha^{\frac{1}{q}} \le (2^q\alpha + 1-\alpha)^{\frac{1}{q}}
	$
	for all $\alpha \in [0,1]$ and $q \in [1,\infty[$. This, however, is false (e.g., let $\alpha = 1/2$ and $q = 2$).
\end{example}
Now we establish that the set $\MMc^\ast(\FFc_2)$ has a maximum element, meaning that there exists $\mu^\ast \in \MMc^\ast(\FFc_2)$ such that $\theta^\ast \preceq \mu^\ast$ for all $\theta^\ast \in \MMc^\ast(\FFc_2)$.

We will need the following proposition, which extends and generalizes a criterion used in \cite[\S{ }3, p. 563]{Bu0} to prove that the upper Buck density of the set of perfect squares is zero.
\begin{proposition}
	\label{prop:modular_criterion}
	Fix $X \subseteq \HHb$, let $(\mu_\ast, \mu^\ast)$ be a conjugate pair on $\HHb$ such that $\mu^\ast$ is an upper quasi-density, and for $k \ge 1$ and $S \subseteq \HHb$ denote by $w_k(S)$ the number of residues $h \in \llb 0, k-1 \rrb$ with the property that
	$\mu^\ast(S \cap (k \cdot \HHb + h)) > 0$. Then, for all $k \in \NNb^+$ we have
	\begin{equation}
	\label{equ:inequality_1}
	1 - \frac{w_k(X^c)}{k} \le \mu_\ast(X) \le \mu^\ast(X) \le \frac{w_k(X)}{k}.
	\end{equation}
\end{proposition}
\begin{proof}
	Pick $k \in \mathbf N^+$ and $S \subseteq \HHb$, and let $\WWc_k(S)$ be the set of all integers $h \in \llb 0, k-1 \rrb$ for which $\mu^\ast(S \cap (k \cdot \HHb + h))> 0$, so that $w_k(S) = |\WWc_k(S)|$.
	
	However we choose $h \in \llb 0, k-1 \rrb$, there exists $S_h \subseteq \HHb$ such that 
	$$
	S \cap (k \cdot \HHb + h) = k \cdot S_h + h. 
	$$
	Therefore, Propositions \ref{prop:elementary_properties_of_d-pairs}\ref{item:prop:elementary_properties_of_d-pairs(iib)} and \ref{prop:mu_of_finite_sets} and axioms  \ref{item:F3} and \ref{item:F4b} yield
	\begin{equation}\label{equ:equ_1}
	\begin{split}
	\mu^\ast(S) & \le \sum_{h=0}^{k-1} \mu^\ast(S \cap (k \cdot \HHb + h)) = \sum_{h \in \WWc_k(S)} \mu^\ast(k \cdot S_h + h) \\
	&  = \sum_{h \in \WWc_k(S)} \frac{1}{k} \mu^\ast(S_h) \le \frac{w_k(S)}{k}.
	\end{split}
	\end{equation}
	Note that here we have used, among other things, that
	$\bigcup_{h=0}^{k-1} (k \cdot \HHb + h) = \HHb \setminus \VVc$ for some finite $\mathcal{V} \subseteq \HHb$ (in particular, $\VVc = \llb 1, k-1 \rrb$ if $\HHb = \NNb^+$, and $\VVc = \emptyset$ otherwise), in such a way that
	$$
	S \setminus \mathcal{V}  = (S \cap \HHb) \setminus \mathcal{V} = \bigcup_{h=0}^{k-1} (S \cap (k \cdot \HHb + h)).
	$$
	Thus, we obtain
	\begin{equation}
	\label{equ:equ_2}
	\mu_\ast(S^c) = 1 - \mu^\ast(S) \ge 1 - \frac{w_k(S)}{k},
	\end{equation}
	and it follows by taking $S = X$ in \eqref{equ:equ_1} and $S = X^c$ in \eqref{equ:equ_2} that
	\begin{equation}
	\label{equ:3}
	1 - \frac{w_k(X^c)}{k} \le \mu_\ast(X)\quad\text{and}\quad \mu^\ast(X) \le \frac{w_k(X)}{k},
	\end{equation}
	which, together with Proposition \ref{prop:elementary_properties_of_d-pairs}\ref{item:prop:elementary_properties_of_d-pairs(vib)},
	implies \eqref{equ:inequality_1}.
\end{proof}
\begin{theorem}
	\label{th:buck_is_maximum}
	Let $\bbf^\ast$ be the upper Buck density on $\HHb$. Then $\bbf^\ast$ is the maximum of both $\MMc^\ast(\FFc)$ and $\MMc^\ast(\FFc_2)$.
\end{theorem}
\begin{proof}
	It is clear that $\MMc^\ast(\FFc) \subseteq \MMc^\ast(\FFc_2)$, and by Example \ref{exa:buck's_measure_density} we have $\bbf^\ast \in \MMc^\ast(\FFc)$. Thus, we just need to show that $\bbf^\ast$ is the maximum of $\MMc^\ast(\FFc_2)$.
	
	To this end, fix $X \subseteq \HHb$. It follows from \cite[Theorem 1]{Pas}, which carries over to the slightly more general version of the upper Buck density considered in the present paper,
	that there exists an increasing sequence $(k_i)_{i \ge 1}$ of positive integers such that
	\begin{equation*}
	\bbf^\ast(X) = \lim_{i \to \infty} \frac{r_{k_i}(X)}{k_i},
	\end{equation*}
	where for $k \ge 1$ we write $r_k(X)$ for the number of residues $h \in \llb 0, k-1 \rrb$ with the property that $X \cap (k \cdot \HHb + h) \ne \emptyset$. So Propositions \ref{prop:modular_criterion} and \ref{prop:mu_of_finite_sets} yield that, for every $\mu^\ast \in \MMc^\ast(\FFc_2)$,
	$$
	\mu^\ast(X) \le \liminf_{k \to \infty} \frac{r_k(X)}{k} \le \lim_{i \to \infty} \frac{r_{k_i}(X)}{k_i} = \bbf^\ast(X),
	$$
	which confirms, since $X$ was arbitrary, that $\bbf^\ast$ is the maximum element of $\MMc^\ast(\FFc_2)$.
\end{proof}
\begin{remark}\label{rem:buck_is_maximum}
	Let $\bbf^\ast$ be the upper Buck density on $\mathbf H$, and let $Y \subseteq X \subseteq \HHb$ and $\mu^\ast \in \MMc^\ast(\FFc_2)$. We get from Remark \ref{rem:upper_quasi-densities_are_non-negative}, Theorem \ref{th:buck_is_maximum}, and the monotonicity of $\bbf^\ast$ (Proposition \ref{prop:generalized_Buck's}) that $0 \le \mu^\ast(Y) \le \bbf^\ast(Y) \le \bbf^\ast(X)$. Hence, $\bbf^\ast(X) = 0$ implies $\mu^\ast(Y) = 0$.
\end{remark}
In fact, Theorem \ref{th:buck_is_maximum} can be ``dualized'' to show that $\MMc_\ast(\FFc)$ has also a minimum element, i.e., there is $\mu_\ast \in \MMc_\ast(\FFc)$ such that $\mu_\ast \preceq \theta_\ast$ for all $\theta_\ast \in \MMc_\ast(\FFc_2)$. 
\begin{lemma}
	\label{lem:reversing_order_by_duality}
	Let $(\lambda_\ast, \lambda^\ast)$ and $(\mu_\ast, \mu^\ast)$ be conjugate pairs on $\bf H$. Then $\lambda^\ast \preceq \mu^\ast$ if and only if $\mu_\ast \preceq \lambda_\ast$.
\end{lemma}
\begin{proof}
	Just recall that $\lambda_\ast(X) = 1 - \lambda^\ast(X^c)$ and $\mu_\ast(X) = 1 - \mu^\ast(X^c)$ for every $X \subseteq \mathbf H$.
\end{proof}
\begin{corollary}
	\label{cor:dual_Buck_is_minimal}
	Let $\bbf_\ast$ be the lower Buck density on $\HHb$. Then $\bbf_\ast$ is the minimum of both $\MMc_\ast(\FFc)$ and $\MMc_\ast(\FFc_2)$.
\end{corollary}
\begin{proof}
	Since $\MMc_\ast(\FFc) \subseteq \MMc_\ast(\FFc_2)$ and $\bbf_\ast \in \MMc_\ast(\FFc)$ by Example \ref{exa:buck's_measure_density}, it is enough to prove that $\bbf_\ast$ is the minimum of $\MMc_\ast(\FFc_2)$. But this is immediate by Theorem \ref{th:buck_is_maximum} and Lemma \ref{lem:reversing_order_by_duality}.
\end{proof}
The next result is a generalization of \cite[Theorem 2]{Pas} and a straightforward consequence of Theorem \ref{th:buck_is_maximum} and Corollary \ref{cor:dual_Buck_is_minimal}.
\begin{corollary}
	\label{cor:pasteka_squeeze_on_densities}
	If $\mu: \mathcal{P}(\HHb) \pto \RRb$ is a quasi-density on $\HHb$, then $\dom(\bbf) \subseteq \dom(\mu)$ and $\mu(X) = \bbf(X)$  for every $X \in \dom(\bbf)$, where $\bbf$ is the Buck density on $\HHb$.
\end{corollary}
The question of the existence of minimal elements of $\MMc^\ast(\FFc)$ and $\MMc^\ast(\FFc_2)$ is subtler, where a  function $\mu^\ast \in \MMc^\ast(\GGc)$ is minimal if there does not exist any $\theta^\ast \in \MMc^\ast(\GGc)$ for which $\theta^\ast \prec \mu^\ast$. The answer is negative in any model of ZF that admits two or more additive upper densities, and hence in ZFC (see Remark \ref{rem:additivity}):
In fact, let $\theta^\ast$ be an additive upper density on $\HHb$, and suppose for a contradiction that $\theta^\ast$ is not a minimal element of $\MMc^\ast(\FFc_2)$, i.e., there exists $\mu^\ast \in \MMc^\ast(\FFc_2)$ with $\mu^\ast \prec \theta^\ast$. Then Proposition \ref{prop:elementary_properties_of_d-pairs}\ref{item:prop:elementary_properties_of_d-pairs(vib)} and Lemma \ref{lem:reversing_order_by_duality} yield $\theta_\ast \preceq \mu_\ast \prec \theta^\ast$, where $\theta_\ast$ is the lower dual of $\theta^\ast$. This is however impossible, because 
$\theta^\ast = \theta_\ast$.

With this said, we can show that $\MMc^\ast(\GGc)$ and $\MMc_\ast(\GGc)$ have at least another notable structural property related to the order $\preceq$. But first we need some terminology.

Specifically, we let a \textit{complete upper semilattice} be a pair $(L, \le_L)$ consisting of a set $L$ and a (partial) order $\le_L$ on $L$ such that, for every non-empty subset $S$ of $L$, the set
$$
\Lambda(S) := \{y \in L: x \le_L y\text{ for all }x \in S\},
$$
has a least element, namely, there exists $y_0 \in \Lambda(S)$ such that $y_0 \le_L y$ for every $y \in \Lambda(S)$; cf., e.g., \cite[\S\S{ }1.10 and 3.14]{Gra}, where the condition $S \ne \emptyset$ is not assumed.
On the other hand, we say that $(L, \le_L)$ is a complete lower semilattice if and only if $(L, \ge_L)$ is a complete upper semilattice, where $\ge_L$ is the partial order on $L$ defined by taking $x \ge_L y$ if and only if $y \le_L x$.
\begin{proposition}
	\label{prop:upper_semilattice}
	$(\MMc^\ast(\GGc), \preceq)$ is a complete upper semilattice.
\end{proposition}
\begin{proof}
	Pick a non-empty subset $S$ of $\MMc^\ast(\GGc)$, and let $\theta^\ast$ denote the function $
	\PPc(\HHb) \to \RRb: X \mapsto \sup_{\mu^\ast \in S} \mu^\ast(X)$, which is well defined because $S \ne \emptyset$ and the image of each function in $S$ is contained in $[0,1]$.
	
	It is clear that $\imag(\theta^\ast) \subseteq [0,1]$ and $\mu^\ast \preceq \theta^\ast$ for all $\mu^\ast \in S$. In addition, if $\mu^\ast(\HHb) = 1$ for every $\mu^\ast \in \MMc^\ast(\GGc)$ then $\theta^\ast(\HHb) = 1$, and if every $\mu^\ast \in \MMc(\GGc)$ is subadditive then
	$$
	\theta^\ast(X \cup Y) \le  \sup_{\mu^\ast \in S} (\mu^\ast(X) + \mu^\ast(Y)) \le \sup_{\mu^\ast \in S} \mu^\ast(X) + \sup_{\mu^\ast \in S} \mu^\ast(Y) = \theta^\ast(X) + \theta^\ast(Y),
	$$
	for all $X, Y \subseteq \HHb$ (i.e., $\mu^\ast$ is subadditive too).
	Similarly, $\theta^\ast$ is monotone, $(-1)$-homogeneous, or translational invariant, respectively, if so is every $\mu^\ast \in \MMc^\ast(\GGc)$ (we omit details).
\end{proof}
Incidentally, Proposition \ref{prop:upper_semilattice} implies that both $\MMc^\ast(\FFc)$ and $\MMc^\ast(\FFc_2)$ have the maximum element (not necessarily the same), but does not identify it more precisely, in contrast to Theorem \ref{th:buck_is_maximum}. Similar considerations apply also to the following result, when compared with Corollary \ref{cor:dual_Buck_is_minimal}.
\begin{corollary}
	$(\MMc_\ast(\GGc), \preceq)$ is a complete lower semilattice.
\end{corollary}
Finally, let $S$ be a set and $(A, \le_A)$ a directed preordered set. To wit, $A$ is a set and $\le_A$ is a reflexive and transitive binary relation (i.e., a preorder) on $A$ such that for every non-empty finite $B \subseteq A$ there is $\alpha \in A$ with $\beta \le_A \alpha$ for all $\beta \in B$. A net $(f_\alpha)_{\alpha  \in  A}$ of functions $S \to \RRb$ is any function $\eta: A \to \mathbf{R}^S$.
We say that the net $(f_\alpha)_{\alpha  \in  A}$ is pointwise convergent if there exists a function $f: S \to \RRb$, which we call a pointwise limit of $(f_\alpha)_{\alpha  \in  A}$, such that for every $x \in S$ the real net $(f_\alpha(x))_{\alpha \in A}$ converges to $f(x)$
in the usual topology on $\RRb$.

With these definitions and the above notation in place, we have the following:
\begin{proposition}
	\label{prop:pointwise_limits}
	Let $(\mu_\alpha)_{\alpha  \in  A}$ and $(\lambda_\alpha)_{\alpha  \in  A}$ be, respectively, pointwise convergent nets with values in $\MMc^\ast(\GGc)$ and $\MMc_\ast(\GGc)$, and denote by $\mu$ a pointwise limit of $(\mu_\alpha)_{\alpha  \in  A}$ and by $\lambda$ a pointwise limit of $(\lambda_\alpha)_{\alpha  \in  A}$. Then $\mu$ and $\lambda$ are uniquely determined and belong, respectively, to $\MMc^\ast(\GGc)$ and $\MMc_\ast(\GGc)$.
\end{proposition}
We conclude the section by adding one more distinguished item to our list of upper densities:
\begin{example}
	\label{exa:polya}
	Let $\mathfrak{p}^\ast$ be the upper P\'olya density on $\HHb$, viz. the function
	$$
	\PPc(\HHb) \to \RRb: X \mapsto \lim_{s \to 1^-} \limsup_{n \to \infty} \frac{|X \cap [ns,n]|}{(1-s)n}.
	$$
	It is not difficult to check that the dual of $\mathfrak p^\ast$ is the function
	$$
	\PPc(\HHb) \to \RRb: X \mapsto \lim_{s \to 1^-} \liminf_{n \to \infty} \frac{|X \cap [ns,n]|}{(1-s)n},
	$$
	which we refer to as the lower P\'olya density on $\HHb$. Among other things, $\mathfrak p^\ast$ has found a number of remarkable applications in analysis and economic theory (see, e.g., \cite{Pol}, \cite{Lev},
	\cite{Leo17}, and \cite{Mar}), but what is perhaps more interesting in the frame of the present work is that $\mathfrak p^\ast$ is an upper density in the sense of our definitions: This follows from Proposition \ref{prop:pointwise_limits} and the fact that $\mathfrak p^\ast$ is the pointwise limit of the real net of the upper $\alpha$-densities on $\HHb$, see \cite[Theorem 4.3]{LMS}.
\end{example}

\section{Closing remarks}
\label{sec:closing_remarks}
Below, we draw a list of questions we have not been able to answer, some of them being broad generalizations of questions from the literature on densities.
\begin{question}
	\label{open:density_sets}
	Let $\mu^\ast$ be an upper quasi-density on $\mathbf H$ and $\mu_\ast$ its conjugate, and for every $X \subseteq \HHb$ denote by $\mathfrak{D}_X(\mu^\ast)$ the set of all pairs $(a_1, a_2) \in \RRb^2$ such that $a_1 = \mu_\ast(Y)$ and $a_2 = \mu^\ast(Y)$ for some $Y \subseteq X$. Is $\mathfrak{D}_X(\mu^\ast)$ a convex or closed subset of $\RRb^2$ for every $X \subseteq \HHb$?
\end{question}
Notice that, if $\mu^\ast$ is an upper density, then $\mathfrak{D}_X(\mu^\ast)$ is contained, by Proposition \ref{prop:elementary_properties_of_d-pairs}\ref{item:prop:elementary_properties_of_d-pairs(vib)}, in the trapezium
$\{(a_1, a_2) \in [0,1]^2: 0 \le a_1 \le \mu_\ast(X)\ \text{and}\ a_1 \le a_2 \le \mu^\ast(X)\}$,
but this is no longer the case when $\mu^\ast$ does not satisfy \ref{item:F2}, as it follows from Theorem \ref{th:independence_of_(F2)}.

Actually, the answer to Question \ref{open:density_sets} is positive when $\mu^\ast$ is, for some real exponent $\alpha \ge -1$, the classical upper $\alpha$-density on $\NNb^+$, as essentially proved in \cite{GrekThes, Grek78}. More in general, the same is true for certain upper weighted densities, as we get from \cite[Theorem 2]{GMT}.

On a related note, we ask the following question, which has a positive answer for the Buck density and the Banach density, see \cite[Theorem 2.1]{PaSa} and \cite[Theorem 4.2]{GLS}, respectively.
\begin{question}
	\label{open:Darboux_for_densities}
	Let $\mu$ be a quasi-density on $\HHb$. Given $X \in  \dom(\mu)$ and $a \in [0, \mu(X)]$, does there exist $Y \in \dom(\mu)$ such that $Y \subseteq X$ and $\mu(Y) = a$?
\end{question}
It is perhaps worth mentioning that, if $\mu^\ast$ is an upper quasi-density on $\HHb$ and $X \subseteq Y \subseteq \HHb$, then for every $a \in [\mu^\ast(X), \mu^\ast(Y)]$ there exists $A \in \mathcal P(\HHb)$ with $X \subseteq A \subseteq Y$ and $\mu^\ast(A) = a$, see \cite[Theorem 1]{LT17}: This provides a positive answer to a weaker version of Question \ref{open:Darboux_for_densities}.

In fact, the very existence of upper quasi-densities that are not upper densities (Theorem \ref{th:independence_of_(F2)}) raises a number of problems.
In particular, Remark \ref{rem:buck_is_maximum} suggests the following:
\begin{question}
	\label{open:subsets_of_sets_of_zero_density}
	Let $\mu^\ast$ be an upper quasi-density on $\HHb$. If $X \subseteq \HHb$ and $\mu^\ast(X) = 0$, can there exist a set  $Y \subseteq X$ for which $\mu^\ast(Y) \ne 0$?
	And is it possible that $\mu^\ast(X) \ne \mu^\ast(Y)$ for some $X,Y \subseteq \HHb$ such that $|X \triangle Y| < \infty$?
\end{question}
The case of upper densities is covered by the following supplement to Proposition \ref{prop:elementary_properties_of_d-pairs}\ref{item:prop:elementary_properties_of_d-pairs(vb)}:
\begin{proposition}
	\label{prop:invariance_under_unions_with_finite_sets}
	Let $\mu^\ast$ be an upper density on $\HHb$ and $\mu_\ast$ its conjugate, and pick $X, Y \subseteq \HHb$. If $|X \triangle Y| < \infty$, then $\mu^\ast(X) = \mu^\ast(Y)$ and $\mu_\ast(X) = \mu_\ast(Y)$.
\end{proposition}
\begin{proof}
	Assume that $|X \triangle Y| < \infty$. By Proposition \ref{prop:mu_of_finite_sets}, we have $\mu^\ast(X \triangle Y) = 0$, and since $X \triangle Y = X^c \triangle Y^c$, this implies, along with Proposition \ref{prop:null_invariance}\ref{prop:null_invariance(i)}, that $\mu^\ast(X) = \mu^\ast(Y)$ and $\mu^\ast(X^c) = \mu^\ast(Y^c)$. Therefore, we find
	$
	\mu_\ast(X) = 1 - \mu^\ast(X^c) = 1 - \mu^\ast(Y^c) = \mu_\ast(Y)$.
\end{proof}
A special case of Question \ref{open:subsets_of_sets_of_zero_density}, which would simplify some of the proofs of this paper and show that the case $\HHb = \NNb^+$ can be reduced, to some degree, to the case $\HHb = \NNb$, is as follows:
\begin{question}
	\label{open:unique_extension}
	Can an upper quasi-density on $\NNb^+$ be \textit{uniquely} extended to an upper quasi-density on $\NNb$?
\end{question}
We obtain from Propositions \ref{prop:null_invariance}\ref{prop:null_invariance(iv)} and \ref{prop:mu_of_finite_sets} that if we replace upper quasi-densities with upper densities in Questions \ref{open:subsets_of_sets_of_zero_density} and \ref{open:unique_extension}, then the answer to the former is negative and the answer to the latter is positive. The same argument leads to the next proposition (we omit details):
\begin{proposition}
	\label{prop:translational_invariance_of_lower_densities}
	Let $(\mu_\ast, \mu^\star)$ be a conjugate pair on $\HHb$ such that $\mu^\ast$ satisfies \ref{item:F2}, \ref{item:F3} and \ref{item:F5}. Then $\mu_\star$ is translational invariant.
\end{proposition}
On the other hand, we have not succeeded in settling the following:
\begin{question}
	\label{open:low_quasi-density_is_shift_invariant}
	Is it true that every lower density is $(-1)$-homogeneous? Is there a lower quasi-density that is not translational invariant?
\end{question}
On a different note, one may wonder if a set $X \subseteq \mathbf H$ such that $\mu^\ast(X) > 0$ for some upper density $\mu^\ast$, has to contain a finite arithmetic progression of length $n$ for some $n \ge 3$ (or even for every $n \ge 3$). But the answer is negative:
It follows from \cite[Theorem 3.2]{PaSa} that the upper Buck density of the set $X := \{n + n!: n \in \NNb\}$ is $1$, yet $X$ does not contain any finite arithmetic progression of length $3$.

Loosely speaking, this means that neither Szemer\'edi's theorem \cite{Szem} nor Roth's theorem \cite{Roth1, Roth2} are ``characteristic of the theory of upper densities'', as long as the notion of upper density is interpreted in the lines of the present work. So it could be interesting to tackle the following:

\begin{question}
	\label{quest:roth_szemeredi}
	Does there exist a reasonable set of axioms, alternative to or sharper than \ref{item:F1}-\ref{item:F5}, for which an ``abstract version'' of Roth's or Szemer\'edi's theorem can be proved?
\end{question}
\textit{Note added in proof.} It turns out that Question \ref{open:unique_extension} has an affirmative answer. In fact, let $\mu^\ast$ be an upper [quasi-]density on $\mathbf N^+$. Then it is relatively easy to verify that the function
$$
\bar{\mu}^\ast: \mathcal{P}(\mathbf{N}) \to \mathbf{R}: X \mapsto \mu^\ast(X+1)
$$
is an upper [quasi-]density on $\mathbf{N}$; and that any extension of $\mu^\ast$ to a shift-invariant function $f: \mathcal{P}(\mathbf{N}) \to \mathbf{R}$ does coincide with $\bar{\mu}^\ast$, on account of the fact that, for every $X \subseteq \mathbf{N}$, one has $X+1 \subseteq \mathbf{N}^+$, and hence $f(X) = f(X+1) = \mu^\ast(X+1) = \bar{\mu}^\ast(X+1) = \bar{\mu}^\ast(X)$.
\section*{Acknowledgments}
The authors are grateful to Georges Grekos (Universit\'e de St-\'Etienne, FR) for his valuable advice; to Christopher O'Neill (UC Davis, US) for fruitful discussions about the independence of axiom \ref{item:F2}; to Carlo Sanna (Universit\`a di Torino, IT) for a careful proofreading;
to Joseph~H.~Silverman (Brown University, US)
and Martin Sleziak (Comenius University, SK)
for useful comments on \href{http://mathoverflow.net/}{MathOverflow.net}; and to an anonymous referee for suggestions that helped improving the overall presentation of the paper.

\end{document}